\documentclass[11pt,reqno,oneside]{amsart}

\usepackage[a4paper,total={6.6in,9.9in}]{geometry}
\usepackage{amsmath,amssymb,amsfonts,amsthm,mathtools,mathrsfs}
\usepackage{exscale}
\usepackage{latexsym}
\usepackage{graphicx}
\usepackage{subfigure}
\usepackage{enumerate}
\usepackage{cite}
\usepackage{xcolor}
\usepackage{hyperref}
\hypersetup{
  colorlinks=true,
  linkcolor={blue!80!black},
  urlcolor={blue!80!black},
  citecolor={blue!80!black}
}

\numberwithin{equation}{section}

\newtheorem{definition}{Definition}[section]
\newtheorem{theorem}{Theorem}[section]
\newtheorem{lemma}[theorem]{Lemma}
\newtheorem{proposition}[theorem]{Proposition}
\theoremstyle{remark}
\newtheorem{remark}[theorem]{Remark}

\title[Weyl-type law for Maxwell transmission eigenfunctions]{A Weyl-type law for surface-localized eigenfunctions of the Maxwell transmission problem}
\author{Yan Jiang}
\address{Department of Mathematics, City University of Hong Kong, Hong Kong SAR, China.}
\email{yjian24@cityu.edu.hk}

\author{Hongyu Liu}
\address{Department of Mathematics, City University of Hong Kong, Hong Kong SAR, China.}
\email{hongyu.liuip@gmail.com, hongyliu@cityu.edu.hk}

\author{Kai Zhang}
\address{Department of Mathematics, Jilin University, Changchun, Jilin, China.}
\email{zhangkaimath@jlu.edu.cn}

\author{Haoran Zheng}
\address{School of Mathematical Sciences, South China Normal University, Guangzhou, Guangdong, China.}
\email{hrzheng@scnu.edu.cn}

\date{}

\begin{document}
	
\maketitle
\noindent{\bf Abstract:}~~
This work studies surface-localized transmission eigenfunctions for time-harmonic electromagnetic scattering. Prior results on this phenomenon are mostly qualitative, proving existence without quantifying distribution. Here we provide a quantitative analysis for the Maxwell transmission eigenvalue problem, covering both TE and TM modes. We first establish a Weyl asymptotic law for the full counting function, with cubic growth with respect to the radius $R$. We then prove Weyl-type upper and lower bounds for the counting functions restricted to surface-localized modes, and demonstrate that  they share the same growth order three. To our knowledge, this is the first quantitative result of its kind, demonstrating that surface-localized eigenfunctions form a non-negligible fraction of the high-frequency spectrum.




\section{Introduction}
Wave localization constitutes a pervasive phenomenon in wave physics, typically arising when propagation is constrained by geometric features, material inhomogeneities, or boundary effects. Canonical examples include confinement induced by external potentials, such as the quantum harmonic oscillator \cite{Simon79}, as well as disorder-driven Anderson localization \cite{Anderson58,YSH23}. Beyond these classical mechanisms, it is now firmly established that spatial confinement of wave energy can also manifest in the absence of external potentials, stemming purely from geometric or interface-induced effects \cite{DE95,LM07}. A central problem at the intersection of physics and spectral theory is to elucidate how such localization emerges intrinsically from the governing equations.

Following the definition set forth in \cite{NG13}, a function $u$ defined on a domain $\Omega \subset \mathbb{R}^d$ is called $L^p$-localized (for $p \geq 1$) if there exists a bounded subset $\Omega_0 \subset \Omega$ that supports the bulk of the $L^p$-norm of $u$, i.e.,
$$
\frac{\|u\|_{L^p(\Omega \setminus \Omega_0)}}{\|u\|_{L^p(\Omega)}} \ll 1 \quad \text{and} \quad \frac{\mu_d(\Omega_0)}{\mu_d(\Omega)} \ll 1,
$$
where $\mu_d$ denotes the $d$-dimensional Lebesgue measure (area for $d=2$ and volume for $d=3$). Qualitatively, such a localized function is predominantly concentrated on a spatially compact subset of the domain and assumes negligibly small values on its complement.

For scalar waves governed by the Laplace or Helmholtz operator, localization phenomena have been the subject of extensive investigation in the literature \cite{CTV03,CTV05,FK07,HHN97,HHT09}. High-frequency modes tend to concentrate near domain boundaries, giving rise to whispering gallery modes, whereas geometric irregularities—such as narrow channels, sharp corners, or waveguide junctions—can effectively trap energy and yield localized eigenstates \cite{GN13, NG13}. These effects have been observed experimentally and analyzed mathematically across diverse settings, including quantum waveguides, optical cavities, and acoustic resonators. A central theme emerging from these studies is that \emph{boundaries and interfaces serve as effective confining structures}, shaping the spatial distribution of eigenfunctions without the need for explicit external potentials.

Motivated by the geometric structure of the eigenfunctions of the Laplace operator, the present paper investigates surface-localization phenomena for interior transmission eigenfunctions arising in time-harmonic electromagnetic (EM) scattering from inhomogeneous media. Let \(\Omega\subset\mathbb{R}^3\) be a bounded Lipschitz domain, and let \(\lambda\in L^\infty(\Omega)\) denote the electric permittivity, with \(\operatorname*{ess\,inf}_{\Omega}\lambda>0\). In the rigorous analysis that follows, we restrict our attention to the constant-coefficient setting \(\lambda=\xi^2\), where \(\xi^2\neq 1\). The interior transmission problem then seeks pairs \((\mathbf E_j,\mathbf H_j)\in H(\operatorname{curl},\Omega)\times H(\operatorname{curl},\Omega)\), \(j=1,2\), such that
\begin{equation}\label{eq:trans1}
	\left\{\begin{array}{lll}
		\operatorname{curl} \mathbf{E}_1 - \mathrm{i}k\mathbf{H}_1 = \mathbf{0}, & \operatorname{curl} \mathbf{H}_1 + \mathrm{i}k\lambda\mathbf{E}_1 = \mathbf{0}, \quad &\text{in } \Omega \\
		\operatorname{curl} \mathbf{E}_2 - \mathrm{i}k\mathbf{H}_2 = \mathbf{0}, & \operatorname{curl} \mathbf{H}_2 + \mathrm{i}k\mathbf{E}_2 = \mathbf{0}, \quad &\text{in } \Omega \\
		\nu \times \mathbf{E}_1 = \nu \times \mathbf{E}_2, & \nu \times  \mathbf{H}_1 = \nu \times  \mathbf{H}_2, \quad &\text{on } \partial\Omega,
	\end{array}\right.
\end{equation}
where \(\nu\) denotes the exterior unit normal to \(\partial\Omega\), \(k\in\mathbb{R}_+\) is the wavenumber of the propagating electromagnetic wave, and \(\mathrm{i} := \sqrt{-1}\). The system \eqref{eq:trans1} admits the trivial solution \(\mathbf{E}_j \equiv \mathbf{H}_j \equiv 0\), \(j=1,2\). If nontrivial solutions \((\mathbf{E}_j,\mathbf{H}_j)_{j=1,2}\) exist, then \(k\in\mathbb{R}_+\) is called a (real) \textit{transmission eigenvalue}, and \(\mathbf{E}_j,\mathbf{H}_j\) are the associated \textit{transmission eigenfunctions}. Since \(\operatorname{curl} \mathbf{E}_j = \mathrm{i}k\mathbf{H}_j\), eliminating the magnetic fields \(\mathbf{H}_j\), \(j=1,2\), from \eqref{eq:trans1} yields the following reduced formulation for \(\mathbf{E}_j\in H(\operatorname{curl}^2,\Omega)\), \(j=1,2\):
\begin{equation}\label{eq:trans2}
	\left\{\begin{array}{lll}
		\operatorname{curl} \operatorname{curl} \mathbf{E}_1 - k^2\lambda\mathbf{E}_1 = \mathbf{0}, & \operatorname{div}(\lambda\mathbf{E}_1) = \mathbf{0} \quad &\text{in } \Omega, \\
		\operatorname{curl} \operatorname{curl} \mathbf{E}_2 - k^2\mathbf{E}_2 = \mathbf{0}, & \operatorname{div} \mathbf{E}_2 = \mathbf{0} \quad &\text{in } \Omega, \\
		\nu \times \mathbf{E}_1 = \nu \times \mathbf{E}_2, & \nu \times (\operatorname{curl} \mathbf{E}_1) = \nu \times (\operatorname{curl} \mathbf{E}_2) \quad &\text{on } \partial\Omega.
	\end{array}\right.
\end{equation}

Recent studies have revealed nontrivial geometric features of transmission eigenfunctions, including enhanced concentration near the boundary of the scattering medium. Under the standard sign-definite contrast hypotheses, the transmission spectrum is known to be discrete (see \cite{CCH16,LC12}). In the constant-radial setting adopted below, the explicit TE/TM characteristic equations yield infinitely many positive real transmission eigenvalues with no finite accumulation point. Numerical evidence in \cite{CDHLW21} further indicates that a substantial portion of the associated eigenfunctions exhibit surface localization, with their $L^2$-energy concentrating in a neighborhood of $\partial\Omega$. The spectral density of such surface-localized eigenmodes was recently studied for the acoustic eigenvalue problem in \cite{JLZZ25}. In the present paper, we consider the Maxwell transmission eigenvalue problem \eqref{eq:trans2} for radially symmetric media with constant electric permittivity, as derived in \cite{MS12}, and establish an estimate on the geometric structure of the corresponding transmission eigenfunctions.

From a physical standpoint, the transmission interface constitutes a region of intense wave--medium interaction, wherein waves undergo reflection, refraction, and mode conversion as a consequence of material contrast. It is thus natural to expect that, under appropriate conditions, the interface itself may serve as an effective trapping mechanism, leading to pronounced wave-energy concentration in its immediate vicinity. This \emph{interface-induced localization} is conceptually analogous to boundary localization for Laplacian eigenfunctions; however, it arises here from coupling between distinct media rather than from purely geometric confinement.

Rigorous results on such surface-localization phenomena are available for several transmission models. For the Helmholtz-type problem
\begin{equation*}
	\Delta w + k^2 \mathbf{n}^2 w = 0, \quad \Delta v + k^2 v = 0, \quad \text{in } \Omega; \quad w = v, \quad \partial_\nu w = \partial_\nu v, \quad \text{on } \partial \Omega,
\end{equation*}
where \(\mathbf{n} = \sqrt{\lambda}\) denotes the refractive index,  it was shown in \cite{CDHLW21} that, along high-frequency sequences $k_m \to \infty$, either ${w_m}$ or ${v_m}$ exhibits surface localization depending on whether $\mathbf{n}>1$ or $0<\mathbf{n}<1$. Remarkably, Deng et al. established the existence of a \emph{double surface localization} phenomenon, whereby both components concentrate in a neighbourhood of the boundary \cite{DJLZ22}. Subsequent work has extended this effect to a broad class of models, including electromagnetic, acoustic--elastic, and purely elastic transmission problems \cite{DLWW22,DLLT23,JLZZ23}.

A central question arising from these observations extends beyond the mere existence of surface-localized transmission eigenmodes to their quantitative spectral prevalence. Specifically, obtaining a precise count—or asymptotic density—of such modes in the spectrum poses a substantial mathematical challenge. For self-adjoint operators such as the Laplacian, questions of this nature are addressed via microlocal analysis, semiclassical measures, and Weyl-type asymptotics, furnishing asymptotic counts of eigenfunctions with specific localization properties. For Maxwell transmission eigenvalue problems, however, the situation is markedly different. The intrinsic non-self-adjointness, complex spectrum, and absence of eigenfunction orthogonality all preclude the direct use of classical methods from spectral geometry.

Further complicating the analysis is the structure of the Maxwell system itself: the unknown vector fields involve the curl–curl operator and are subject to divergence constraints along with boundary transmission conditions. These features render both the spectral analysis and the qualitative study of eigenfunctions significantly more delicate. Hence, whether one can meaningfully quantify the share of surface-localized modes in the full spectrum remains a largely open problem.

Building on the acoustic spectral density analysis in \cite{JLZZ25}, we address the corresponding problem for Maxwell transmission problem by studying their surface-localized eigenfunctions behavior. Our analysis, which covers both TE and TM polarizations, yields a Weyl-type asymptotic law for the full counting function and establishes matching upper and lower cubic-order bounds for the counting functions of surface-localized modes. These results show, in particular, that such modes attain the same leading-order spectral count as the full transmission spectrum, offering a quantitative characterization of their prevalence.

The statistical prevalence of surface-localized Maxwell transmission eigenmodes established in this work indicates that boundary concentration is a robust spectral phenomenon rather than an exceptional occurrence, carrying significant physical implications. Since electromagnetic energy is frequently confined near material interfaces, surface effects tend to govern scattering and resonant behavior across extended frequency intervals. This observation has direct practical implications for radar scattering and stealth technology, where the activation or suppression of surface-dominated modes can markedly alter detectability and imaging performance \cite{CDHLW21}. Moreover, the strong boundary sensitivity of these modes provides a rigorous basis for nondestructive evaluation and material parameter retrieval, since even minor perturbations of surface geometry or material properties lead to observable spectral variations \cite{HCS14}. From a photonic standpoint, the prevalence of surface-trapped transmission eigenmodes parallels known surface-guided resonances such as surface plasmons, facilitating enhanced wave confinement and light–matter coupling along interfaces. This connection underscores the relevance of transmission eigenfunctions to emerging photonic transport and sensing platforms \cite{KMTKMS,RZPLPDNSS13,WCJS08}, and positions them as a key mechanism for interface-driven electromagnetic localization.

The paper is organized as follows.
We begin in Section~\ref{Preliminaries} by introducing the notation and stating the principal results. Section~\ref{sec:weyl} provides the proof of Theorem~\ref{thm:weyl}. The lower and upper bounds are proved in Sections~\ref{sec:lower_bound} and~\ref{sec:upper_bound}, respectively. Finally, Section~\ref{sec:numerics} illustrates the theoretical findings through numerical simulations.


\section{Preliminaries and main results}\label{Preliminaries}

To begin with, we define precisely what we mean by a surface-localized function.
\begin{definition}\label{def:localization}
Given $\varepsilon \in (0,1)$ and $\delta > 0$, a nontrivial vector field $\mathbf{u}$ defined on a domain $\Omega \subset \mathbb{R}^3$ is called \emph{$(\varepsilon,\delta)$-localized near the surface} if
\begin{equation*}
\mathcal{L}_\delta(\mathbf{u}):=\frac{\|\mathbf{u}\|_{L^2(\Omega_\delta)}}{\|\mathbf{u}\|_{L^2(\Omega)}}>1-\varepsilon,
\end{equation*}
where
\begin{equation*}
\Omega_\delta := \{ x \in \Omega : \operatorname{dist}(x,\partial\Omega) < \delta \}.
\end{equation*}
\end{definition}

\begin{remark}
	The quantity $\mathcal{L}_\delta(\mathbf{u})$ measures the proportion of $L^2$-energy contained in a boundary layer of thickness $\delta$. It is invariant under scaling of $\mathbf{u}$ and provides a normalized measure of surface concentration.
\end{remark}


\subsection{Quantitative framework and main results}

In this work, we aim to quantitatively characterize the occurrence of surface-localized transmission eigenfunctions in the high-frequency regime. Therefore, we formulate the problem in terms of counting functions for both TE and TM modes, and then establish Weyl-type asymptotic laws. We first introduce the counting functions for TE modes. The TM counting functions are defined in the same way.

For each fixed $m \geq 1$, let $\mathcal K_m^{\mathrm{TE}}$ denote the complete set of TE interior transmission eigenvalues of angular degree $m$, ordered increasingly:
\begin{equation*}
\mathcal K_m^{\mathrm{TE}} = \{ k_{m,n}^{\mathrm{TE}} : n \geq 1 \},
\end{equation*}
where $k_{m,n}^{\mathrm{TE}}$ is the $n$-th eigenvalue in the set. The azimuthal index, ranging from $-m$ to $m$, is suppressed since the radial eigenvalues are independent of it. Set
\begin{equation*}
\mathcal K^{\mathrm{TE}}=\bigcup_{m\geq1}\mathcal K_m^{\mathrm{TE}}.
\end{equation*}

Let $R>0$ denote the counting threshold. For the unit ball, every modal eigenvalue associated with angular degree $m$ has angular multiplicity $2m+1$, corresponding to the azimuthal indices $-m,\cdots,m$. Hence the full TE counting function is defined by
\begin{equation*}
N^{\mathrm{TE}}(R):=\sum_{m=1}^{\infty} (2m+1) \,
\#\{\, n \geq 1 : 0 < k_{m,n}^{\mathrm{TE}} < R \,\}.
\end{equation*}
Thus, $N^{\mathrm{TE}}(R)$ counts the TE eigenvalue contributions below $R$, weighted by their angular multiplicities. If the same numerical eigenvalue occurs for different angular degrees, the corresponding modal contributions are counted separately.
	
For $k = k_{m,n}^{\mathrm{TE}} \in \mathcal K_m^{\mathrm{TE}}$, let $\bigl( \mathbf E_{1,m,n}^{\mathrm{TE}}, \mathbf E_{2,m,n}^{\mathrm{TE}} \bigr)$
be a nontrivial TE electric-field transmission eigenpair of \eqref{eq:trans2} associated with this eigenvalue. We say that this modal eigenpair is surface-localized provided that
\begin{equation*}
\max_{j=1,2} \mathcal L_\delta \bigl( \mathbf E_{j,m,n}^{\mathrm{TE}} \bigr) > 1 - \varepsilon,
\end{equation*}
for some $\varepsilon \in (0,1)$ and $\delta > 0$. This condition is independent of the normalization of the eigenfields, as $\mathcal L_\delta$ is invariant under multiplication by a nonzero scalar. The azimuthal index is suppressed here, since for each fixed $j$, the $2m+1$ azimuthal copies of $\mathbf E_{j,m,n}^{\mathrm{TE}}$ share the same radial profile and thus the same localization ratio.

The localized TE counting function is then given explicitly by
\begin{equation*}
N_{\mathrm{loc}}^{\mathrm{TE}}(R;\varepsilon,\delta):=
\sum_{m=1}^{\infty}(2m+1)\sum_{\substack{k\in\mathcal K_m^{\mathrm{TE}}\\0<k<R}}
\mathbf 1 \left\{ \max_{j=1,2} \mathcal L_\delta \bigl(\mathbf E_{j,m,k}^{\mathrm{TE}}\bigr)
>1-\varepsilon
\right\}.
\end{equation*}
Thus, \(N_{\mathrm{loc}}^{\mathrm{TE}}(R;\varepsilon,\delta)\) counts the TE modal contributions below \(R\) whose eigenfunctions are localized near the boundary, weighted by their angular multiplicity.

The TM quantities are defined by replacing \(\mathrm{TE}\) with \(\mathrm{TM}\) throughout. In particular, from the definition above, for \(\sigma \in \{\mathrm{TE}, \mathrm{TM}\}\), we have
\begin{equation*}
0\leq N_{\mathrm{loc}}^{\sigma}(R;\varepsilon,\delta) \leq N^{\sigma}(R), \quad
\sigma\in\{\mathrm{TE},\mathrm{TM}\}.
\end{equation*}
For each fixed \(R\), only finitely many angular degrees contribute; the corresponding angular cutoff is established in Section~\ref{sec:weyl}. 

Our main results determine the high-frequency growth of the full counting functions and give quantitative bounds for the localized counting functions.

\begin{theorem}\label{thm:weyl}
Let \(\Omega = B_1(0) \subset \mathbb{R}^3\), let \(\lambda \in (0,1)\) be as in \eqref{eq:trans2}, and define \(\xi := \sqrt{\lambda}\). Then the transmission eigenvalue counting functions satisfy
\begin{equation*}
N^{\mathrm{\sigma}}(R)=\frac{2 R^3}{9 \pi}\left(1-\xi^3\right)+\mathcal{O}(R^2),
\quad R\to\infty, \quad \sigma \in\{\mathrm{TE},\mathrm{TM}\}.
\end{equation*}
\end{theorem}

\begin{theorem}[Lower bound]\label{thm:lower_bdd}
Under the assumptions of Theorem~\ref{thm:weyl}, for any $0<\delta<1$ and $0<\varepsilon<\frac{1}{2}$, the localized counting functions satisfy
\begin{equation*}
	N_{\mathrm{loc}}^{\sigma}(R;\varepsilon,\delta)
	\geqslant 
	\frac{2 R^3}{9 \pi}\left(1-\xi^2\right)^{3/2}
	- C R^2 \log R, \quad R\to\infty, \quad \sigma \in\{\mathrm{TE},\mathrm{TM}\}.
\end{equation*}
\end{theorem}

\begin{theorem}[Upper bound]\label{thm:upper_bdd}
Under the assumptions of Theorem~\ref{thm:weyl}, fix \(0<\delta<1\) and \(0<\varepsilon<1-\sqrt{\delta}\). Choose any \(\eta\) such that $
2\varepsilon-\varepsilon^2<\eta<1-\delta$, and define
\begin{equation}\label{eq:upper_a_definition}
	a:=\sqrt{\frac{(1-\delta)^2-\eta^2}{1-\eta^2}} \in (0,1).
\end{equation}
Then the localized counting functions satisfy
\begin{equation*}
N^{\sigma}_{\mathrm{loc}}\left(R;\varepsilon,\delta\right)
\leqslant \,\frac{2}{9\pi} \left[(1-\xi^2 a^2)^{3/2}-\xi^3(1-a^2)^{3/2}\right]R^3+CR^2\log R, \quad R\to\infty, \quad \sigma \in\{\mathrm{TE},\mathrm{TM}\}.
\end{equation*}
\end{theorem}

The above results show that surface-localized transmission eigenvalues satisfy a Weyl-type asymptotic law. In particular, the leading-order growth of the localized counting function is proportional to that of the full spectrum, implying that such modes occupy a non-vanishing portion of the high-frequency regime. Below, we rigorously establish Theorems \ref{thm:weyl}--\ref{thm:upper_bdd} in the radially symmetric setting, whereas for the general case we restrict ourselves to numerical evidence.

\begin{remark}
We primarily focus on the case $0<\xi<1$. The results for $\xi>1$ can be obtained via the transformation
\begin{equation*}
k^* = k \xi, \quad \xi^* = \frac{1}{\xi}, \quad \mathbf{E}_1^* = \mathbf{E}_1, \quad \mathbf{E}_2^* = \mathbf{E}_2.
\end{equation*}
Under this change of variables, the transmission eigenvalue problem \eqref{eq:trans2} takes the form
\begin{equation*}
\left\{
\begin{aligned}
	&\operatorname{curl} \operatorname{curl} \mathbf{E}_1^* - (k^*)^2 (\xi^*)^2 \mathbf{E}_1^* = 0, 
	\quad &&\operatorname{div}\bigl((\xi^*)^2 \mathbf{E}_1^*\bigr) = 0, 
	&&\text{in } \Omega, \\[0.4em]
	&\operatorname{curl} \operatorname{curl} \mathbf{E}_2^* - (k^*)^2 \mathbf{E}_2^* = 0, 
	\quad &&\operatorname{div} \mathbf{E}_2^* = 0, 
	&&\text{in } \Omega, \\[0.4em]
	&\nu \times \mathbf{E}_1^* = \nu \times \mathbf{E}_2^*, 
	\quad &&\nu \times \bigl(\operatorname{curl} \mathbf{E}_1^*\bigr) = \nu \times \bigl(\operatorname{curl} \mathbf{E}_2^*\bigr), 
	&&\text{on } \partial \Omega.
\end{aligned}
\right.
\end{equation*}
This system retains the same structural form as the original problem \eqref{eq:trans2}. Consequently, the conclusions of Theorems \ref{thm:weyl}--\ref{thm:upper_bdd} remain valid for the case $\xi>1$.
\end{remark}


\subsection{Radial decomposition and modal structure}

In this subsection, we collect the necessary structural properties of transmission eigenfunctions for the Maxwell system \eqref{eq:trans2} in the radially symmetric setting. Although these results have been derived in \cite{MS12,DLWW22}, we provide a different proof based on vector spherical harmonics. Some of the technical ingredients will be needed in our subsequent analysis, and we therefore present them below.

Without loss of generality, assume that \(\Omega\) is the unit ball in \(\mathbb{R}^3\), i.e.
\begin{equation*}
\Omega = \{\mathbf{x} \in \mathbb{R}^3 : |\mathbf{x}| < 1 \}.
\end{equation*}
It follows from \cite{CK19} that transmission eigenfunctions admit a separation-of-variables representation:
\begin{equation}\label{eq:Fourier_decompostion}
	\mathbf{E}_1(\mathbf{x})
	= \sum_{m=1}^{\infty}\sum_{l=-m}^{m}
	\left(\alpha_m^l \mathbf{M}_m^l(\mathbf{x})
	+ \beta_m^l \mathbf{N}_m^l(\mathbf{x})\right),
	\quad
	\mathbf{E}_2(\mathbf{x})
	= \sum_{m=1}^{\infty}\sum_{l=-m}^{m}
	\left(\widetilde{\alpha}_m^l \widetilde{\mathbf{M}}_m^l(\mathbf{x})
	+ \widetilde{\beta}_m^l \widetilde{\mathbf{N}}_m^l(\mathbf{x})\right),
\end{equation}
where the vector spherical wave functions are defined by
\begin{equation*}
\begin{aligned}
	\mathbf{M}_m^l(\mathbf{x})
	&= \operatorname{curl} \left\{\mathbf{x} j_m\bigl(k \xi|\mathbf{x}|\bigr) Y_m^l(\hat{\mathbf{x}})\right\},
	\quad
	\mathbf{N}_m^l(\mathbf{x})
	= \frac{1}{\mathrm{i} k} \operatorname{curl} \mathbf{M}_m^l(\mathbf{x}), \\
	\widetilde{\mathbf{M}}_m^l(\mathbf{x})
	&= \operatorname{curl}\left\{\mathbf{x} j_m(k|\mathbf{x}|) Y_m^l(\hat{\mathbf{x}})\right\},
	\quad
	\widetilde{\mathbf{N}}_m^l(\mathbf{x})
	= \frac{1}{\mathrm{i} k} \operatorname{curl} \widetilde{\mathbf{M}}_m^l(\mathbf{x}).
\end{aligned}
\end{equation*}
Here, $j_m$ denotes the spherical Bessel function of order $m$, $Y_m^l$ denotes the spherical harmonic of degree $m$ and order $l$, and $\hat{\mathbf{x}} = \mathbf{x}/|\mathbf{x}|$.

It can be verified that
\begin{equation*}
\operatorname{curl} \operatorname{curl} \mathbf{u} = k^2 \xi^2 \mathbf{u},
\quad \operatorname{div} \mathbf{u} = 0,
\end{equation*}
for \(\mathbf{u} = \mathbf{M}_m^l, \mathbf{N}_m^l\). Similarly,
\begin{equation*}
\operatorname{curl} \operatorname{curl} \widetilde{\mathbf{u}} = k^2 \widetilde{\mathbf{u}},
\quad \operatorname{div} \widetilde{\mathbf{u}} = 0,
\end{equation*}
for \(\widetilde{\mathbf{u}} = \widetilde{\mathbf{M}}_m^l, \widetilde{\mathbf{N}}_m^l\). Furthermore, the fields \(\mathbf{M}_m^l\) and \(\widetilde{\mathbf{M}}_m^l\) correspond to TE modes, whereas \(\mathbf{N}_m^l\) and \(\widetilde{\mathbf{N}}_m^l\) correspond to TM modes.

We have the following characterization for TE transmission eigenvalues.
\begin{proposition}[TE modes]\label{prop:TE}
Let \(\Omega = \{\mathbf{x}\in\mathbb{R}^3:|\mathbf{x}|<1\}\), and let \(0<\xi=\sqrt{\lambda}<1\) be as in \eqref{eq:trans2}. Then the TE transmission eigenvalues are precisely the positive real numbers \(k>0\) such that
\begin{equation}\label{eq:F_m^TE(k)}
F_m^{\mathrm{TE}}(k):= k j_m(\xi k)\, j_m'(k) - \xi k j_m(k)\, j_m'(\xi k) = 0,
\end{equation}
for some \(m\in\mathbb{N}_+\).
\end{proposition}

\begin{proof}
Fix a modal pair $(m,l)$ with $m \geq 1$ and $-m \leq l \leq m$. For TE modes, we set $\beta_m^l = \widetilde{\beta}_m^l = 0$ in \eqref{eq:Fourier_decompostion}, so that
\begin{equation*}
	\mathbf{E}_1 = \alpha_m^l \mathbf{M}_m^l,	\quad
	\mathbf{E}_2 = \widetilde{\alpha}_m^l \widetilde{\mathbf{M}}_m^l.
\end{equation*}
Using classical identities for vector spherical harmonics (see \cite{AS72}), one obtains the following representation. Indeed, since $Y_m^l$ depends only on the angular variable, one has
$\mathbf{x}\times\nabla Y_m^l=\hat{\mathbf{x}}\times\nabla_{\mathbb{S}^2}Y_m^l$, which is independent of $r = |\mathbf{x}|$. Thus, up to an inessential common sign that can be absorbed into the modal coefficients,
\begin{equation}\label{eq:eigenfunction_TE}
	\mathbf{M}_m^l = j_m(k\xi r)\boldsymbol{\Phi}_m^l, 	\quad
	\widetilde{\mathbf{M}}_m^l = j_m(k r)\boldsymbol{\Phi}_m^l,	\quad r=|\mathbf{x}|,
\end{equation}
where the vector spherical harmonics are defined by
\begin{equation*}
	\boldsymbol{\Phi}_m^l = \mathbf{x} \times \nabla Y_m^l.
\end{equation*}

Imposing the transmission conditions on $|\mathbf{x}| = 1$, we obtain from the tangential continuity condition $\nu \times \mathbf{E}_1 = \nu \times \mathbf{E}_2$ that
\begin{equation*}
	\alpha_m^l j_m(k\xi) = \widetilde{\alpha}_m^l j_m(k).
\end{equation*}
Moreover, the standard identities for vector spherical harmonics (cf. \cite{AS72}) give the tangential curl traces
\begin{eqnarray*}
	\nu \times (\operatorname{curl}\mathbf{M}_m^l)\big|_{r=1}
	&=& \big(j_m(k\xi)+k\xi j_m'(k\xi)\big)\boldsymbol{\Phi}_m^l,\\
	\nu \times (\operatorname{curl}\widetilde{\mathbf{M}}_m^l)\big|_{r=1}
	&=& \big(j_m(k)+k j_m'(k)\big)\boldsymbol{\Phi}_m^l.
\end{eqnarray*}
Thus the second transmission condition yields
\begin{equation*}
\alpha_m^l\big(j_m(k\xi)+k\xi j_m'(k\xi)\big)=\widetilde{\alpha}_m^l\big(j_m(k)+k j_m'(k)\big).
\end{equation*}
Together with $\alpha_m^l j_m(k\xi)=\widetilde{\alpha}_m^l j_m(k)$, the associated linear system is singular, i.e.,
\begin{equation*}
	\det \begin{pmatrix}
		j_m(k\xi) & -j_m(k)\\
		j_m(k\xi)+k\xi j_m'(k\xi) & -\big(j_m(k)+k j_m'(k)\big)
	\end{pmatrix}=0.
\end{equation*}
A direct expansion of the determinant reduces it to the conclusion (\ref{eq:F_m^TE(k)}).
\end{proof}

Accordingly, the counting function admits the decomposition
\begin{equation}\label{eq:N_TE_decomposition}
	N^{\mathrm{TE}}(R) = \sum_{m=1}^{\infty} N_m^{\mathrm{TE}}(R),
\end{equation}
where
\begin{equation*}
	N_m^{\mathrm{TE}}(R) := (2m+1)\,\# \left\{ k \in (0,R) \;\middle|\; F_m^{\mathrm{TE}}(k)=0 \right\}.
\end{equation*}

\begin{remark}\label{rem:M}
The series in \eqref{eq:N_TE_decomposition} converges for each $R>0$. More precisely, as shown in \cite{PS14}, there exists $M(R,\xi)\in\mathbb{N}$ such that
\begin{equation*}
N_m^{\mathrm{TE}}(R)=0, \quad \text{for all } m > M(R,\xi).
\end{equation*}
\end{remark}

\begin{proposition}[TM modes]\label{prop:TM}
Under the assumptions of Proposition~\ref{prop:TE}, the TM transmission eigenvalues are precisely the positive real numbers $k>0$ such that
\begin{equation}\label{eq:F_m^TM(k)}
F_m^{\mathrm{TM}}(k):= (1 - \xi^2) j_m(x) j_m(\xi x)
+\xi x \left(j_m(x) j_m'(\xi x) - \xi j_m(\xi x) j_m'(x)\right)= 0,
\end{equation}
for some $m\in\mathbb{N}_+$.
\end{proposition}

\begin{proof}
Fix a modal pair $(m,l)$ with $m \geq 1$ and $-m \leq l \leq m$. For TM modes, we set $\alpha_m^l = \widetilde{\alpha}_m^l = 0$ in \eqref{eq:Fourier_decompostion}, so that
\begin{equation*}
\mathbf{E}_1 = \beta_m^l \mathbf{N}_m^l, \quad
\mathbf{E}_2 = \widetilde{\beta}_m^l \widetilde{\mathbf{N}}_m^l.
\end{equation*}
Using the standard vector spherical harmonic identities (up to the fixed normalization of $\boldsymbol{Y}_m^l$ and $\boldsymbol{\Psi}_m^l$), the TM fields may be written as
\begin{eqnarray*}
\mathbf{N}_m^l
&=&
\frac{1}{\mathrm{i} k}
\left(
\frac{m(m+1)}{r} j_m(k\xi r)\boldsymbol{Y}_m^l
+
\frac{1}{r}\frac{\mathrm{d}}{\mathrm{d}r}\big(rj_m(k\xi r)\big)\boldsymbol{\Psi}_m^l
\right), \\
\widetilde{\mathbf{N}}_m^l
&=&
\frac{1}{\mathrm{i} k}
\left(
\frac{m(m+1)}{r} j_m(k r)\boldsymbol{Y}_m^l
+
\frac{1}{r}\frac{\mathrm{d}}{\mathrm{d}r}\big(rj_m(k r)\big)\boldsymbol{\Psi}_m^l
\right),
\end{eqnarray*}
where $r=|\mathbf{x}|$, $\boldsymbol{Y}_m^l=Y_m^l\hat{\mathbf{x}}$ is radial, and $\boldsymbol{\Psi}_m^l=r\nabla Y_m^l$ is tangential.
	
On $\partial\Omega$ one has $\nu \times \boldsymbol{Y}_m^l = 0$, so the condition $\nu \times \mathbf{E}_1 = \nu \times \mathbf{E}_2$ gives only the tangential $\boldsymbol{\Psi}_m^l$ relation
\begin{equation}\label{eq:b_TM_1}
	\beta_m^l\big(j_m(k\xi)+k\xi j_m'(k\xi)\big)
	=
	\widetilde{\beta}_m^l\big(j_m(k)+k j_m'(k)\big).
\end{equation}
The second transmission condition is obtained from $\nu\times\operatorname{curl}\mathbf{E}_1=\nu\times\operatorname{curl}\mathbf{E}_2$. Since
\begin{equation*}
\operatorname{curl}\mathbf{N}_m^l=-\mathrm{i}k\xi^2\mathbf{M}_m^l, \quad
\operatorname{curl}\widetilde{\mathbf{N}}_m^l=-\mathrm{i}k\widetilde{\mathbf{M}}_m^l,
\end{equation*}
this boundary condition gives
\begin{equation}\label{b_TM_2}
	\beta_m^l\xi^2 j_m(k\xi)
	=
	\widetilde{\beta}_m^l j_m(k).
\end{equation}
Equations \eqref{eq:b_TM_1} and \eqref{b_TM_2} form a homogeneous linear system for $(\beta_m^l,\widetilde{\beta}_m^l)$. A nontrivial solution exists if and only if
\begin{equation*}
\det \begin{pmatrix}
j_m(k\xi)+k\xi j_m'(k\xi) & -\big(j_m(k)+k j_m'(k)\big)\\
\xi^2 j_m(k\xi) & -j_m(k)
\end{pmatrix}=0.
\end{equation*}
A direct expansion of the determinant reduces it to the conclusion (\ref{eq:F_m^TM(k)}).
\end{proof}


\section{The proof of Theorem \ref{thm:weyl}}\label{sec:weyl}
This section is concerned with the analysis of the counting functions \(N^{\mathrm{TE}}(R)\) and \(N^{\mathrm{TM}}(R)\) and the derivation of the corresponding Weyl-type asymptotic formulae. The discussion is first carried out for the TE polarization, since the TM polarization can be treated in a completely analogous fashion requiring only straightforward adjustments.


\subsection{TE modes}\label{subsec:TE}

We begin with the characterization obtained in Proposition~\ref{prop:TE}, which states that the TE transmission eigenvalues $k$ are precisely the positive zeros of the family of functions
\begin{equation*}
F_m^{\mathrm{TE}}(x) = x j_m(\xi x) j_m'(x) - \xi x j_m(x) j_m'(\xi x), \quad m \in \mathbb{N}_+.
\end{equation*}

Since \(x>0\), it is convenient to introduce the normalized function
\begin{equation}\label{eq:F_m}
	\widehat F_m^{\mathrm{TE}}(x):=-\frac{1}{x}F_m^{\mathrm{TE}}(x)
	=	\xi j_m(x)j_m'(\xi x)-j_m(\xi x)j_m'(x).
\end{equation}
As the factor \(-1/x\) is nonzero for \(x>0\), the zeros of \(F_m^{\mathrm{TE}}\) on \((0,\infty)\) coincide exactly with those of \(\widehat F_m^{\mathrm{TE}}\). Consequently, these zeros are precisely the solutions \(x>0\) of the equation
\begin{equation}\label{eq:F}
	\xi j_m(x)j_m'(\xi x)=j_m(\xi x)j_m'(x).
\end{equation}

For $x>0$, define
\begin{equation*}
H_m(x):=\frac{j_m'(x)}{j_m(x)},\quad D_m^{\mathrm{TE}}(x):=\xi H_m(\xi x)-H_m(x).
\end{equation*}
Whenever \(j_m(x)j_m(\xi x)\neq 0\), the logarithmic derivatives are well-defined and we have
\begin{equation*}
\widehat F_m^{\mathrm{TE}}(x)
=
j_m(x)j_m(\xi x)D_m^{\mathrm{TE}}(x).
\end{equation*}
Since the prefactor is nonzero precisely under this condition, the zeros of \(\widehat F_m^{\mathrm{TE}}\) in this region are exactly those of \(D_m^{\mathrm{TE}}\). Hence the regular TE eigenvalues are precisely the zeros of \(D_m^{\mathrm{TE}}\).

\begin{lemma}\label{lem:F_m_intersect}
Assume \(0<\xi<1\). If \(x_0>0\) satisfies \(j_m(x_0)j_m(\xi x_0)\neq 0\) and \(D_m^{\mathrm{TE}}(x_0)=0\), then
\begin{equation}\label{eq:DTEm}
\bigl(D_m^{\mathrm{TE}}\bigr)'(x_0)=1-\xi^2>0.
\end{equation}
Consequently, every such zero is simple and, since the derivative is strictly positive, it corresponds to an upward crossing of \(D_m^{\mathrm{TE}}\).
\end{lemma}
\begin{proof}
The spherical Bessel equation is
\begin{equation}\label{eq:Bessel}
	x^2j_m''(x)+2xj_m'(x)+\bigl(x^2-m(m+1)\bigr)j_m(x)=0.
\end{equation}
Dividing by \(j_m(x)\) on any interval where it does not vanish gives the Riccati equation for \(H_m = j_m'/j_m\):
\begin{equation}\label{eq:H'}
	H_m'(x)=-\frac{2}{x}H_m(x)-1+\frac{m(m+1)}{x^2}-H_m(x)^2.
\end{equation}
To compute \((D_m^{\mathrm{TE}})'(x_0)\), note that
\begin{equation*}
	(D_m^{\mathrm{TE}})'(x) = \xi^2 H_m'(\xi x) - H_m'(x).
\end{equation*}
Let \(A := H_m(x_0)\) and \(B := H_m(\xi x_0)\). The condition \(D_m^{\mathrm{TE}}(x_0)=0\) implies \(A = \xi B\). Evaluating (\ref{eq:H'}) at \(x_0\) and \(\xi x_0\), we obtain
\begin{equation*}
H_m'(x_0) = -\frac{2}{x_0}A - 1 + \frac{m(m+1)}{x_0^2} - A^2,
\end{equation*}
and
\begin{equation*}
H_m'(\xi x_0) = -\frac{2}{\xi x_0}B - 1 + \frac{m(m+1)}{\xi^2 x_0^2} - B^2.
\end{equation*}
Substituting these into \((D_m^{\mathrm{TE}})'(x_0)\) and using \(A = \xi B\) yields the conclusion (\ref{eq:DTEm}).
\end{proof}

Let \(0<\rho_1<\rho_2<\cdots\) denote the positive zeros of \(j_m\) (equivalently, of \(J_{m+1/2}\)). For \(n\in\mathbb{N}_+\), the poles of \(D_m^{\mathrm{TE}}\) occur precisely at $x=\rho_n$ (from $H_m(x)$) and at $x=\frac{\rho_n}{\xi}$ (from $H_m(\xi x)$). Thus \(D_m^{\mathrm{TE}}\) has these two families of poles on the positive real axis.

Following \cite[Sec.~4.2]{PS14}, we define the small intervals
\begin{equation*}
I_n=(\rho_n,\rho_{n+1}], \quad n\in\mathbb{N}_+,
\end{equation*}
and the large intervals
\begin{equation*}
L_q=\left(\frac{\rho_q}{\xi},\frac{\rho_{q+1}}{\xi}\right], \quad q\in\mathbb{N}_+.
\end{equation*}
Away from the poles, the factorization in \eqref{eq:F_m} yields
\begin{equation*}
\widehat F_m^{\mathrm{TE}}(x)=j_m(x)j_m(\xi x)D_m^{\mathrm{TE}}(x),
\end{equation*}
and hence the regular TE eigenvalues are precisely the regular zeros of \(D_m^{\mathrm{TE}}\).

If \(\rho\) is a zero of \(j_m\), then
\begin{equation*}
H_m(x)=\frac{1}{x-\rho}-\frac{1}{\rho}+\frac{m(m+1)+3-\rho^2}{3\rho^2}(x-\rho)
+\mathcal O\bigl((x-\rho)^2\bigr).
\end{equation*}
Hence, for a pole \(x=\rho_n\) of \(H_m(x)\) that is not simultaneously a pole of \(H_m(\xi x)\) (i.e., \(\rho_n \neq \rho_q/\xi\) for all \(q\)), the term \(-H_m(x)\) dominates in $D_m^{\mathrm{TE}}$, yielding
\begin{equation*}
\lim_{x\to \rho_n^+} D_m^{\mathrm{TE}}(x) = -\infty, \quad
\lim_{x\to \rho_n^-} D_m^{\mathrm{TE}}(x) = +\infty.
\end{equation*}
Similarly, for a pole \(x=\rho_q/\xi\) of \(H_m(\xi x)\) that is not a zero of \(j_m\), the term \(\xi H_m(\xi x)\) dominates, giving
\begin{equation*}
\lim_{x\to (\rho_q/\xi)^+} D_m^{\mathrm{TE}}(x) = +\infty, \quad
\lim_{x\to (\rho_q/\xi)^-} D_m^{\mathrm{TE}}(x) = -\infty.
\end{equation*}
If the two families of poles coincide, then the leading singular terms cancel. Consequently, \(D_m^{\mathrm{TE}}\) has a removable singularity at \(p\), with the local expansion
\begin{equation}\label{eq:TE_common_pole_expansion}
	D_m^{\mathrm{TE}}(x)=\frac{1-\xi^2}{3}(x-p)	+\mathcal O\bigl((x-p)^2\bigr).
\end{equation}
Since \(0<\xi<1\), the coefficient \((1-\xi^2)/3\) is strictly positive. Thus, after removing the singularity, \(D_m^{\mathrm{TE}}\) has a simple zero at \(p\) with an upward crossing. Moreover, the factorization in \eqref{eq:F_m} shows that \(p\) is also a zero of \(F_m^{\mathrm{TE}}\). At such a point, the TE boundary matrix in Proposition~\ref{prop:TE} has rank one, so the corresponding radial eigenspace is one-dimensional. The full geometric multiplicity of the eigenvalue is therefore \(2m+1\), accounting for the \(2m+1\) azimuthal modes.

We now count the regular zeros of \(D_m^{\mathrm{TE}}\). By Lemma~\ref{lem:F_m_intersect}, every regular zero is simple and upward (negative to positive), so any pole-free interval contains at most one such zero.

Consider \(I_n=(\rho_n,\rho_{n+1}]\) strictly contained in \(L_q=(\rho_q/\xi,\rho_{q+1}/\xi]\). Then \(\rho_n,\rho_{n+1}\) are poles of \(H_m\) only, and
\begin{equation*}
\lim_{x\to\rho_n^+} D_m^{\mathrm{TE}}(x)=-\infty,\quad
\lim_{x\to\rho_{n+1}^-} D_m^{\mathrm{TE}}(x)=+\infty.
\end{equation*}
Thus, by continuity and the upward-crossing property, \(I_n\) contains exactly one regular zero.

If \(I_n\) contains two large-only poles (poles of \(H_m(\xi x)\) not shared with \(H_m\)), then to the right of the first, \(D_m^{\mathrm{TE}}\to+\infty\), while to the left of the second, \(D_m^{\mathrm{TE}}\to-\infty\). Connecting these would force a downward crossing, impossible. Hence such an interval cannot contain two large-only poles. If it contains exactly one large-only pole, the left subinterval has both endpoint limits equal to \(-\infty\), and the right subinterval has both endpoint limits equal to \(+\infty\). Since the endpoint signs match and downward crossings are forbidden, neither subinterval contains a regular zero.

Finally, if \(I_n\) and \(L_q\) share an endpoint, \(\rho_n=\rho_q/\xi\), this is a common pole. Following \cite[Sec.~4.2]{PS14}, assign the corresponding eigenvalue to the half-open interval \(I_n=(\rho_n,\rho_{n+1}]\) to avoid double counting.

Let
\begin{equation*}
B_{m+\frac12}(T):=\#\{x\in(0,T): J_{m+\frac12}(x)=0\},
\end{equation*}
and let \(Z_m^{\mathrm{TE}}(T)\) count the positive zeros of \(F_m^{\mathrm{TE}}\) in \((0,T)\). The expansion
\begin{equation*}
D_m^{\mathrm{TE}}(x)=\frac{1-\xi^2}{2m+3}x+\mathcal O(x^3),\quad x\to0+,
\end{equation*}
shows the sign is positive near the origin; hence upward and downward crossings before \(T\) differ by at most one. Therefore
\begin{equation*}
Z_m^{\mathrm{TE}}(T)=B_{m+\frac12}(T)-B_{m+\frac12}(\xi T)e_m^{\mathrm{TE}}(T),
\quad |e_m^{\mathrm{TE}}(T)|\leq1.
\end{equation*}
Thus the angular-degree \(m\) contribution to the TE counting function is
\begin{equation}\label{N_m_decomposition}
	N_m^{\mathrm{TE}}(T)=
	(2m+1)\left[B_{m+\frac12}(T)-B_{m+\frac12}(\xi T)+e_m^{\mathrm{TE}}(T)\right],
	\quad 	|e_m^{\mathrm{TE}}(T)|\leq1.
\end{equation}
The formula remains valid if \(T\) is a zero or pole, using the appropriate one-sided limit; the endpoint error term is evaluated at the same cutoff \(T\) as the counting function.

\begin{remark}\label{rem:TE}
The function $B_{m+1/2}(R)$ is increasing in $R$ and decreasing in $m$ by the interlacing of Bessel zeros \cite{AS72}.  Standard asymptotics for the positive zeros of \(J_{1/2}\) give \(B_{1/2}(R)=R/\pi+\mathcal O(1)\) \cite[Theorem~7.6.5]{OL74}.
\end{remark}

Since the first zero of \(J_{m+1/2}\) satisfies \(\rho_1>m+1/2\) (see \cite{AS72}), both \(B_{m+1/2}(R)\) and \(B_{m+1/2}(\xi R)\) vanish for \(m\ge\lceil R-1/2\rceil\). As \(D_m^{\mathrm{TE}}\) is positive before its first small pole, these angular orders contribute no TE roots below \(R\), hence \(M(R,\xi)\le\lceil R-1/2\rceil-1=\mathcal O(R)\). Consequently,
\begin{equation*}
\left|\sum_{m=1}^{M(R,\xi)}(2m+1)e_m^{\mathrm{TE}}(R)\right|
\le \sum_{m=1}^{M(R,\xi)}(2m+1)=\mathcal O(R^2).
\end{equation*}

Let \(\omega_3:=|B_1(0)|=4\pi/3\). Weyl's law for the scalar Dirichlet problem \cite[Theorem~1.6.1]{SV97} gives
\begin{equation*}
\sum_{m=0}^{\infty}(2m+1)B_{m+1/2}(R)
=(2\pi)^{-3}\omega_3^2R^3+\mathcal O(R^2).
\end{equation*}
The sum starts at \(m=1\); extending to \(m=0\) costs \(\mathcal O(R)\), while the endpoint errors contribute \(\mathcal O(R^2)\). Therefore, using \eqref{N_m_decomposition},
\begin{equation}\label{N^TE(R)}
	\begin{aligned}
		N^{\mathrm{TE}}(R)
		&=\sum_{m=1}^{M(R,\xi)}(2m+1)\left[B_{m+\frac12}(R)-B_{m+\frac12}(\xi R)+e_m^{\mathrm{TE}}(R)\right]\\
		&=\sum_{m=0}^{\infty}(2m+1)\left[B_{m+\frac12}(R)-B_{m+\frac12}(\xi R)\right]+\mathcal O(R^2)\\
		&=(2\pi)^{-3}\omega_3^2(1-\xi^3)R^3+\mathcal O(R^2)\\
		&=\frac{2R^3}{9\pi}(1-\xi^3)+\mathcal O(R^2).
	\end{aligned}
\end{equation}
This completes the TE part of Theorem~\ref{thm:weyl}.


\subsection{TM modes}

For TM modes, Proposition~\ref{prop:TM} states that the eigenvalues \(k\) are the positive real roots of
\begin{equation}\label{eq:F_TM}
	F_m^{\mathrm{TM}}(x) := (1 - \xi^2) j_m(x) j_m(\xi x) + \xi x \left(j_m(x) j_m^{\prime}(\xi x) - \xi j_m(\xi x) j_m^{\prime}(x)\right).
\end{equation}
Analogous to the TE case,
\begin{equation*}
	N^{\mathrm{TM}}(R) = \sum_{m=1}^{\infty} N_m^{\mathrm{TM}}(R),
\end{equation*}
where
\begin{equation}\label{eq:N_m_TM}
	N^{\mathrm{TM}}_m(R) := (2m+1)\,\# \left\{ k\in (0,R) : F_m^{\mathrm{TM}}(k) = 0 \right\}.
\end{equation}

Define the auxiliary function
\begin{equation*}
G_m^{\mathrm{TM}}(x) = \frac{j_m(x) + x j_m^{\prime}(x)}{j_m(x)}.
\end{equation*}
Using \eqref{eq:F_TM}, we obtain the equivalent form
\begin{equation*}
F^{\mathrm{TM}}_m(x) = j_m(\xi x)j_m(x)\bigl(G_m^{\mathrm{TM}}(\xi x) - \xi^2 G_m^{\mathrm{TM}}(x)\bigr).
\end{equation*}

\begin{lemma}\label{rem:FTM}
For \(0<\xi<1\), at any regular root \(k_0\) with \(j_m(k_0)j_m(\xi k_0)\neq 0\), we have
\begin{equation*}
\left. \frac{d}{dx}\left( G_m^{\mathrm{TM}}(\xi x)-\xi^2G_m^{\mathrm{TM}}(x) \right) \right|_{x=k_0} > 0.
\end{equation*}
\end{lemma}

\begin{proof}
Using \eqref{eq:Bessel}, we obtain
\begin{align*}
	(G_m^{\mathrm{TM}})'(x)
	&= \frac{(x j_m'' + 2j_m')j_m - j_m'(j_m + x j_m')}{j_m^2} \\
	&= \frac{m(m+1)}{x} - x + \frac{1}{x} G_m^{\mathrm{TM}}(x)\bigl(1-G_m^{\mathrm{TM}}(x)\bigr).
\end{align*}
Applying this to the combination \(G_m^{\mathrm{TM}}(\xi x)-\xi^2G_m^{\mathrm{TM}}(x)\) yields
\begin{align*}
&	\left(G_m^{\mathrm{TM}}(\xi x)-\xi^2G_m^{\mathrm{TM}}(x)\right)'
	= \xi\Bigg[
	\frac{m(m+1)}{x}\left(\frac1\xi-\xi\right) \\
	&+ \frac{G_m^{\mathrm{TM}}(\xi x)\bigl(1-G_m^{\mathrm{TM}}(\xi x)\bigr)}{\xi x}
-\frac{\xi}{x}G_m^{\mathrm{TM}}(x)\bigl(1-G_m^{\mathrm{TM}}(x)\bigr)
	\Bigg].
\end{align*}
For a solution \(k_0\) of \(G_m^{\mathrm{TM}}(\xi x)=\xi^2 G_m^{\mathrm{TM}}(x)\), evaluating the derivative at \(x=k_0\) gives
\begin{align*}
	\left(G_m^{\mathrm{TM}}(\xi k_0)-\xi^2G_m^{\mathrm{TM}}(k_0)\right)'
	&= \xi\Bigg[
	\frac{m(m+1)}{\xi k_0}(1-\xi^2)
	+\frac{\xi}{k_0}G_m^{\mathrm{TM}}(k_0)\bigl(1-\xi^2G_m^{\mathrm{TM}}(k_0)\bigr) \\
	&\quad - \frac{\xi}{k_0}G_m^{\mathrm{TM}}(k_0)\bigl(1-G_m^{\mathrm{TM}}(k_0)\bigr)
	\Bigg] \\
	&= \frac{\xi}{k_0}(1-\xi^2)
	\left[\frac{m(m+1)}{\xi}
	+\xi\bigl(G_m^{\mathrm{TM}}(k_0)\bigr)^2\right] > 0,
\end{align*}
which completes the proof.
\end{proof} 

For TM modes, applying the same interval-counting argument as in Subsection~\ref{subsec:TE} yields, for each \(m\ge 1\),
\begin{equation}\label{eq:N_m_TM_decomposition_corrected}
N_m^{\mathrm{TM}}(T)=(2m+1)\left[B_{m+\frac12}(T)-B_{m+\frac12}(\xi T)
+e_m^{\mathrm{TM}}(T)\right], \quad
|e_m^{\mathrm{TM}}(T)|\leq 1.
\end{equation}
As in the TE case, the formula remains valid at poles or regular zeros by taking the appropriate one-sided limit, consistent with the half-open interval \((0,T)\) in the counting function.

For \(R>0\), define the TM angular cutoff by
\begin{equation*} 
M_{\mathrm{cut}}^{\mathrm{TM}}(R,\xi)
:= \max\{m\ge 1 : N_m^{\mathrm{TM}}(R)>0\},
\end{equation*}
with the convention that the maximum is zero if the set is empty. If \(m\ge \lceil R-\frac12\rceil\), then \(m+\frac12\ge R\). Since the first zero \(\rho_1\) of \(J_{m+\frac12}\) satisfies \(\rho_1>m+\frac12\), the corresponding difference function has no poles in \((0,R)\). Combined with the positive initial sign \(D_m^{\mathrm{TM}}(0+)>0\) and the upward-crossing property of regular zeros, this rules out TM eigenvalues in \((0,R)\) for such \(m\). Hence
\begin{equation*}
M_{\mathrm{cut}}^{\mathrm{TM}}(R,\xi)
\le \left\lceil R-\frac12\right\rceil-1 = \mathcal O(R).
\end{equation*}
Consequently,
\begin{equation*}
\sum_{m=1}^{M_{\mathrm{cut}}^{\mathrm{TM}}(R,\xi)}
(2m+1)|e_m^{\mathrm{TM}}(R)|
= \mathcal O(R^2).
\end{equation*}

Summing \eqref{eq:N_m_TM_decomposition_corrected} over admissible \(m\), extending the sum to \(m=0\) (which costs \(\mathcal O(R)\)), and applying the scalar Weyl law gives
\begin{equation}\label{eq:N^TM(R)}
	\begin{aligned}
		N^{\mathrm{TM}}(R)
		&= \sum_{m=1}^{M_{\mathrm{cut}}^{\mathrm{TM}}(R,\xi)}
		(2m+1)\left[B_{m+\frac12}(R)-B_{m+\frac12}(\xi R)+e_m^{\mathrm{TM}}(R)\right] \\
		&= \sum_{m=0}^{\infty}(2m+1)\left[B_{m+\frac12}(R)-B_{m+\frac12}(\xi R)\right] + \mathcal O(R^2) \\
		&= (2\pi)^{-3}\omega_3^2(1-\xi^3)R^3 + \mathcal O(R^2).
	\end{aligned}
\end{equation}
This completes the proof of the TM part of Theorem~\ref{thm:weyl}.


\section{Proof of Theorem \ref{thm:lower_bdd}}\label{sec:lower_bound}


\subsection{TE modes}

In this subsection we derive lower bounds for the localized TE counting function \(N_{\mathrm{loc}}^{\mathrm{TE}}(R;\varepsilon,\delta)\). Recall that a TE modal pair is called \((\varepsilon,\delta)\)-localized near the boundary if
\begin{equation*}
\max_{j=1,2} \mathcal L_\delta(\mathbf E_{j,m,k}^{\mathrm{TE}}) > 1-\varepsilon.
\end{equation*}
Thus a mode is counted whenever at least one of its two electric components is localized. In particular, the subfamily satisfying
\begin{equation*}
\mathcal L_\delta(\mathbf E_{1,m,k}^{\mathrm{TE}}) > 1-\varepsilon
\end{equation*}
is contained in the set counted by \(N_{\mathrm{loc}}^{\mathrm{TE}}(R;\varepsilon,\delta)\), so counting only this subfamily yields a lower bound for the full localized count.

For a TE modal pair, \(\mathbf E_1\) is proportional to the vector spherical wave function \(\mathbf M_m^l\) in \eqref{eq:eigenfunction_TE}. By the orthonormality of the vector spherical harmonics \(\boldsymbol{\Phi}_m^l\), the angular factor cancels in the quotient defining \(\mathcal L_\delta\). Hence, for the component \(\mathbf E_1\), we obtain the radial expression
\begin{equation}\label{energy_E1}
	\mathcal L_\delta^2(\mathbf E_1)
	=
	\frac{\displaystyle\int_{1-\delta}^{1} r^2 j_m^2(k\xi r)\,dr}
	{\displaystyle\int_{0}^{1} r^2 j_m^2(k\xi r)\,dr}
	=
	\frac{\displaystyle\int_{k\xi(1-\delta)}^{k\xi} t J_{m+1/2}^2(t)\,dt}
	{\displaystyle\int_{0}^{k\xi} t J_{m+1/2}^2(t)\,dt}.
\end{equation}

\begin{lemma}\label{lem:key_TE}
For \(0<\xi<1\), \(0<k<\frac{m+1/2}{\xi}\), \(0<\varepsilon<\frac12\), and \(0<\delta<1\), there exists a constant \(C(\varepsilon,\delta)\), independent of \(k,m,\xi\), such that
\begin{equation*}
\mathcal L_\delta(\mathbf E_1) > 1-\varepsilon,
\quad\text{as~} m>C(\varepsilon,\delta).
\end{equation*}
Here, \(C(\varepsilon,\delta)\) is specified in \eqref{eq:C_lowerbound}.
\end{lemma}

\begin{proof} 
Let $\tau=1-\delta$ and $q=k\xi$. Then \(0 < \tau < 1\) and \(0 < q < m + \frac{1}{2}\). Equation \eqref{energy_E1} yields
\begin{equation}\label{eq:TE W_E_1}
	1 - \mathcal{L}^2_{\delta}(\mathbf{E}_1)
	=
	\frac{\displaystyle\int_0^{q\tau} t\,J_{m + \frac{1}{2}}^2(t)\,\mathrm{d}t}
	{\displaystyle\int_0^q t\,J_{m + \frac{1}{2}}^2(t)\,\mathrm{d}t}.
\end{equation}

Define \(P_{m + \frac{1}{2}}(t) = t J_{m + \frac{1}{2}}^2(t)\). Since the first positive zero of \(J_{m + \frac{1}{2}}\) exceeds \(m + \frac{1}{2}\) (\cite{AS72}), the function \(J_{m + \frac{1}{2}}\) is positive on \((0,q]\). Krasikov's logarithmic-derivative estimate \cite[Theorem~1]{KRA06} states that, for \(0 < t \leq m + 1\),
\begin{equation*}
\frac{J_{m + \frac{1}{2}}'(t)}{J_{m + \frac{1}{2}}(t)}\ge
\frac{\sqrt{(2m+2)^2 - 4t^2} - 1}{2t}.
\end{equation*}
Consequently, for all \(0 < t < q\),
\begin{equation}\label{eq:P_log_derivative}
	\frac{\mathrm d}{\mathrm dt}\log P_{m + \frac{1}{2}}(t)
	= \frac{1}{t} + 2\,\frac{J_{m + \frac{1}{2}}'(t)}{J_{m + \frac{1}{2}}(t)}
	\ge \frac{\sqrt{(2m+2)^2 - 4t^2}}{t} > 0.
\end{equation}
Thus \(P_{m + \frac{1}{2}}\) is strictly increasing on \((0,q)\).

Set $a_\tau := \frac{1+\tau}{2}$, $c_\tau := \sqrt{1 - a_\tau^2}$.
Then \(0 < \tau < a_\tau < 1\). Monotonicity of \(P_{m + \frac{1}{2}}\) further gives
\begin{equation*}
\int_0^{q\tau} P_{m + \frac{1}{2}}(t)\,\mathrm{d}t
\leq q\tau \, P_{m + \frac{1}{2}}(q\tau),
\end{equation*}
and
\begin{equation*}
\int_0^q P_{m + \frac{1}{2}}(t)\,\mathrm{d}t
\geq \int_{q a_\tau}^q P_{m + \frac{1}{2}}(t)\,\mathrm{d}t
\geq q(1 - a_\tau) P_{m + \frac{1}{2}}(q a_\tau).
\end{equation*}
For \(q\tau \le t \le q a_\tau\), the inequalities \(q < m + \frac{1}{2}\) and \(t \le q a_\tau\) imply
\begin{equation*}
\sqrt{(2m+2)^2 - 4t^2}
\geq \sqrt{(2m+2)^2 - 4\bigl(m + \tfrac{1}{2}\bigr)^2 a_\tau^2}
\geq 2\bigl(m + \tfrac{1}{2}\bigr) c_\tau.
\end{equation*}
Integrating \eqref{eq:P_log_derivative} from \(q\tau\) to \(q a_\tau\) therefore yields the estimate, uniform for all \(0 < q < m + \frac{1}{2}\),
\begin{equation}\label{eq:J_mz}
	\frac{P_{m + \frac{1}{2}}(q\tau)}{P_{m + \frac{1}{2}}(q a_\tau)}
	\leq
	\left(\frac{\tau}{a_\tau}\right)^{2(m + \frac{1}{2}) c_\tau}.
\end{equation}
Combining the last three estimates with \eqref{eq:TE W_E_1}, we obtain
\begin{equation*}
1 - \mathcal{L}_{\delta}^2(\mathbf{E}_1)\leq
\frac{\tau}{1 - a_\tau}\left(\frac{\tau}{a_\tau}\right)^{2(m + \frac{1}{2}) c_\tau}.
\end{equation*}
		
For \(m \ge 1\), one has \(2m+1 \ge \frac{3}{2}(m+1)\). Define
\begin{equation}\label{eq:def_f}
	f(x) :=	\left(\frac{2x}{1+x}\right)^{	\frac{3}{2}\sqrt{1-\left(\frac{1+x}{2}\right)^2}},
	\quad 0 < x < 1.
\end{equation}
The base in \eqref{eq:def_f} lies strictly between zero and one and its exponent is positive; hence \(0 < f(x) < 1\) for \(0 < x < 1\). Since \(a_\tau = (1+\tau)/2\), the preceding bound becomes
\begin{equation}\label{eq:TE_uniform_localization_bound}
	1 - \mathcal{L}_{\delta}^2(\mathbf{E}_1)
	\leq
	\frac{2\tau}{1-\tau} \, f(\tau)^{\,m+1}.
\end{equation}

Let \(d_\varepsilon := 2\varepsilon - \varepsilon^2 > 0\) and choose
\begin{equation}\label{eq:C_lowerbound}
	C(\varepsilon,\delta)=\max\left\{1,\,	\frac{\log\!\bigl(d_\varepsilon(1-\tau)/(2\tau)\bigr)}{\log f(\tau)}-1\right\}.
\end{equation}
 The constant in \eqref{eq:C_lowerbound} depends only on \(\varepsilon\) and \(\delta\), and is independent of \(q\), \(k\), \(m\), and \(\xi\). Because \(\log f(\tau) < 0\), the strict inequality \(m > C(\varepsilon,\delta)\) implies
\begin{equation*}
f(\tau)^{m+1} < \frac{d_\varepsilon(1-\tau)}{2\tau}.
\end{equation*}
It follows from \eqref{eq:TE_uniform_localization_bound} that
\begin{equation*}
1 - \mathcal{L}_{\delta}^2(\mathbf{E}_1) < d_\varepsilon = 1 - (1-\varepsilon)^2.
\end{equation*}
Since both \(\mathcal{L}_{\delta}(\mathbf{E}_1)\) and \(1-\varepsilon\) are nonnegative, the above inequality is equivalent to
\(\mathcal{L}_{\delta}(\mathbf{E}_1) > 1-\varepsilon\).
\end{proof}

Recalling that the TE eigenvalues are the positive zeros of \(F_m^{\mathrm{TE}}\), define the \(\mathbf E_1\)-localized modal subcount by
\begin{equation*}
S_m^{\mathrm{TE}}(T;\varepsilon,\delta)
:= (2m+1) \, \#\Bigl\{ k \in (0,T) : F_m^{\mathrm{TE}}(k)=0,\; \mathcal L_\delta(\mathbf E_1)>1-\varepsilon \Bigr\}.
\end{equation*}
The discussion preceding Lemma~\ref{lem:key_TE} shows that this is a subcount of the full localized counting function. Thus
\begin{equation*}
N_{\mathrm{loc}}^{\mathrm{TE}}(R;\varepsilon,\delta)
\ge N_{\mathrm{loc},1}^{\mathrm{TE}}(R;\varepsilon,\delta)
:= \sum_{m=1}^{\infty} S_m^{\mathrm{TE}}(R;\varepsilon,\delta),
\end{equation*}
and \(S_m^{\mathrm{TE}}(R;\varepsilon,\delta) \le N_m^{\mathrm{TE}}(R)\). The exact sum starts at \(m=1\). Moreover, the angular cutoff proved in Section~\ref{sec:weyl} shows that, with
\begin{equation*}
M_2 := \left\lfloor R - \frac12 \right\rfloor,
\end{equation*}
all terms with \(m > M_2\) vanish. Hence
\begin{equation}\label{N_TE_c_decomposition}
	N_{\mathrm{loc},1}^{\mathrm{TE}}(R;\varepsilon,\delta)
	= \sum_{m=1}^{M_2} S_m^{\mathrm{TE}}(R;\varepsilon,\delta).
\end{equation}

Now, we prove the lower bound for the TE modes in Theorem~\ref{thm:lower_bdd}. 
Set $m_0:=\lfloor C(\varepsilon,\delta)\rfloor+1$ and
$T_m:=\min\left\{R,\frac{m+\frac12}{\xi}\right\}$. Then $m\geq m_0$ implies the strict inequality $m>C(\varepsilon,\delta)$.  Every eigenvalue counted by $N_m^{\mathrm{TE}}(T_m)$ satisfies $k<T_m\leq(m+\frac12)/\xi$, so Lemma~\ref{lem:key_TE} gives
\begin{equation*}
S_m^{\mathrm{TE}}(T_m;\varepsilon,\delta)
=N_m^{\mathrm{TE}}(T_m).
\end{equation*}
Using the monotonicity of $S_m^{\mathrm{TE}}$ in its cutoff and \eqref{N_TE_c_decomposition}, we therefore obtain
\begin{equation}\label{eq:TE_localized_subfamily_bound}
N_{\mathrm{loc}}^{\mathrm{TE}}(R;\varepsilon,\delta)
\geq\sum_{m=m_0}^{M_2}N_m^{\mathrm{TE}}(T_m).
\end{equation}

For this choice of $T_m$,
\begin{equation*}
\xi T_m=\min\left\{\xi R,m+\frac12\right\}\leq m+\frac12.
\end{equation*}
Since the first positive zero of $J_{m + \frac{1}{2}}$ is larger than $m + \frac{1}{2}$, one has $B_{m + \frac{1}{2}}(\xi T_m)=0$.  Formula~\eqref{N_m_decomposition} consequently gives
\begin{equation*}
N_m^{\mathrm{TE}}(T_m)
=(2m+1)\left[B_{m+\frac12}(T_m)+e_m^{\mathrm{TE}}(T_m)\right].
\end{equation*}
Because $|e_m^{\mathrm{TE}}(T_m)|\leq1$ and $M_2=\mathcal O(R)$, the sum of the weighted endpoint errors is $\mathcal O(R^2)$.  Now put
\begin{equation*}
M_1:=\left\lfloor\xi R-\frac12\right\rfloor,
\end{equation*}
and the integer split is exact: $T_m=(m+\frac12)/\xi$ for $m\leq M_1$, whereas $T_m=R$ for $m\geq M_1+1$.  Substitution into \eqref{eq:TE_localized_subfamily_bound} yields
\begin{align}
	N_{\mathrm{loc}}^{\mathrm{TE}}(R;\varepsilon,\delta)
	&\geq \sum_{m=m_0}^{M_1}(2m+1)
	B_{m+\frac12}\!\left(\frac{m+\frac12}{\xi}\right) \notag \\
	&\quad + \sum_{m=\max\{m_0,\,M_1+1\}}^{M_2}(2m+1)
	B_{m+\frac12}(R)-C R^2. \label{eq:Nc_estimate_bound}
\end{align} 
Here and below, $C>0$ denotes a generic constant independent of $R$; it may depend on the fixed parameters $(\varepsilon,\delta,\xi)$ and may change from line to line.
According to \cite{JLZZ25}, one has
\begin{equation}\label{eq:NB1}
	B_{m+\frac{1}{2}}\left(\frac{m+\frac{1}{2}}{\xi}\right)
	= \frac{m+\frac{1}{2}}{\pi}
	\left(\frac{\sqrt{1-\xi^2}}{\xi}-\arccos(\xi)\right)
	+ \mathcal{O}(\log R),
\end{equation}
and
\begin{equation}\label{eq:NB2}
	B_{m+\frac{1}{2}}(R)
	= \frac{m+\frac{1}{2}}{\pi}
	\left(\sqrt{\frac{R^2}{\left(m+\frac{1}{2}\right)^2}-1}
	-\arccos\left(\frac{m+\frac{1}{2}}{R}\right)\right)
	+ \mathcal{O}(\log R).
\end{equation}
Consequently, \eqref{eq:Nc_estimate_bound}, \eqref{eq:NB1}, and \eqref{eq:NB2} give
\begin{equation*}
	\begin{aligned}
		N^{\mathrm{TE}}_{\mathrm{loc}}(R;\varepsilon,\delta) \geqslant&  \sum_{m=0}^{M_1}\left(2m+1\right) \frac{m+\frac{1}{2}}{\pi}\left(\frac{\sqrt{1-\xi^2}}{\xi}-\arccos (\xi)\right)\\
		&+ \sum_{m=M_1+1}^{M_2} \left(2m+1\right) \frac{m+\frac{1}{2}}{\pi}\left(\sqrt{\frac{R^2}{\left(m+\frac{1}{2}\right)^2}-1}-\arccos \left(\frac{m+\frac{1}{2}}{R}\right)\right)-CR^2\log  R.
	\end{aligned}
\end{equation*}

It remains to evaluate the two leading sums. Put
\begin{equation*}
\Phi(x)=x\sqrt{1-x^2}-x^2\arccos x,\quad \Phi(1)=0 .
\end{equation*}
Since $\Phi\in C^1([\xi,1])$, the standard Riemann-sum estimate gives
\begin{equation*}
\sum_{m=M_1+1}^{M_2}\Phi\!\left(\frac{m+\frac12}{R}\right)
=R\int_\xi^1\Phi(x)\,dx+\mathcal O(1).
\end{equation*}
Together with
\begin{equation*}
\sum_{m=0}^{M_1}\left(m+\frac12\right)^2
=\frac{(\xi R)^3}{3}+\mathcal O(R^2),
\end{equation*}
we obtain
\begin{equation*}
	N_{\mathrm{loc}}^{\mathrm{TE}}(R;\varepsilon,\delta)
	\geq \frac{2(\xi R)^3}{3\pi}\left(\frac{\sqrt{1-\xi^2}}{\xi}-\arccos\xi	\right)
	+\frac{2R^3}{\pi}	\int_\xi^1\Phi(x)\,dx-C R^2 \log R . 
\end{equation*}
An elementary integration gives
\begin{equation*}
\int_\xi^1\Phi(x)\,dx
=\frac{1}{9}\left[3\xi^3\arccos\xi+(1-4\xi^2)\sqrt{1-\xi^2}\right].
\end{equation*}
Substituting this identity, the leading terms simplify to $\frac{2R^3}{9\pi}(1-\xi^2)^{3/2}$. Therefore
\begin{equation*}
	N_{\mathrm{loc}}^{\mathrm{TE}}(R;\varepsilon,\delta)
	\geq 	\frac{2R^3}{9\pi}(1-\xi^2)^{3/2}-CR^2\log R .
\end{equation*}
The required localization estimate follows from \eqref{eq:J_mz}, while the
zero-counting estimates are given in \eqref{eq:NB1}--\eqref{eq:NB2}.  This
proves the TE lower bound in Theorem~\ref{thm:lower_bdd}.


\subsection{TM modes}
Similar to the TE case, we again focus exclusively on $\mathbf{E}_1$ and derive an estimate for the localization quantity $\mathcal{L}_{\delta}(\mathbf{E}_1)$ associated with TM modes.
\begin{lemma}\label{lem:key_TM}
For every $0<\varepsilon<\frac12$ and $0<\delta<1$, there exists an integer
$M_{\mathrm{loc}}^{\mathrm{TM}}(\varepsilon,\delta)$ with the following
property: for every $0<\xi<1$, every
$m\ge M_{\mathrm{loc}}^{\mathrm{TM}}(\varepsilon,\delta)$, and every TM
transmission eigenvalue $k$ of angular order $m$ satisfying
\[
0<k<\frac{m+\frac12}{\xi},
\]
the first component of each corresponding TM eigenpair satisfies
\[
\mathcal{L}_{\delta}(\mathbf{E}_1)>1-\varepsilon.
\]
In particular, the threshold is independent of $k$ and $\xi$.
\end{lemma}
\begin{proof}
Put $q = k\xi$ and $\tau = 1 - \delta$. Then $0 < q < m + \frac12$ and $0 < \tau < 1$. By Proposition~\ref{prop:TM},
\begin{equation*}
\mathbf N_m^l=\frac{1}{\mathrm i k}
\left(\frac{m(m+1)}{r} j_m(qr)\,\boldsymbol Y_m^l+\frac{1}{r}\frac{\mathrm d}{\mathrm dr}\bigl(r j_m(qr)\bigr)\,\boldsymbol\Psi_m^l\right),
\end{equation*}
where the common constant factor does not affect the localization quotient.
Normalize $Y_m^l$ by $\int_{\mathbb S^2}|Y_m^l|^2\,\mathrm dS = 1$. Since the radial and tangential terms are pointwise orthogonal and
\begin{equation*}
\int_{\mathbb S^2}|\boldsymbol\Psi_m^l|^2\,\mathrm dS = m(m+1),
\end{equation*}
the factor $r^2$ in the volume element cancels the factor $r^{-2}$ in both
terms. Therefore
\begin{equation}\label{eq:TM W_E_1}
\mathcal L_\delta^2(\mathbf E_1)
=\frac{\displaystyle\int_\tau^1\left[m^2(m+1)^2j_m^2(qr)
+m(m+1)\bigl(j_m(qr)+qrj_m'(qr)\bigr)^2\right]\,\mathrm dr}
{\displaystyle\int_0^1\left[m^2(m+1)^2j_m^2(qr)
+m(m+1)\bigl(j_m(qr)+qrj_m'(qr)\bigr)^2\right]\,\mathrm dr}.
\end{equation}
	
For positive numbers $a_i, b_i(i=1,2)$, one has
\begin{equation*}
\frac{a_1+a_2}{b_1+b_2} \leq \max_{i=1,2} \frac{a_i}{b_i}.
\end{equation*}
Applying this to the complementary quotient in \eqref{eq:TM W_E_1} gives
\begin{equation*}
1 - \mathcal L_\delta^2(\mathbf E_1)
\leq
\max\{R_\tau^{(j)}, R_\tau^{(d)}\},
\end{equation*}
where
\begin{equation*}
R_\tau^{(j)}
=\frac{\displaystyle\int_0^\tau j_m^2(qr)\,\mathrm dr}
{\displaystyle\int_0^1 j_m^2(qr)\,\mathrm dr},\quad
R_\tau^{(d)}
=\frac{\displaystyle\int_0^\tau
\bigl(j_m(qr) + qr\, j_m'(qr)\bigr)^2\,\mathrm dr}
{\displaystyle\int_0^1
\bigl(j_m(qr) + qr\, j_m'(qr)\bigr)^2\,\mathrm dr}.
\end{equation*}
	
We estimate these two ratios uniformly for \(0<q<m+\frac12\).  Set
\begin{equation*}
a_\tau=\frac{1+\tau}{2}, \quad b_\tau=\frac{1+a_\tau}{2}, \quad c_\tau=\sqrt{1-a_\tau^2}.
\end{equation*}
Then \(0<\tau<a_\tau<b_\tau<1\).  The first positive zeros of
\(J_{m+\frac12}\) and \(J_{m+\frac32}\) exceed their orders \cite{AS72}; hence
both functions are positive on \((0,m+\frac12)\).  Moreover, the
logarithmic-derivative estimate used in \eqref{eq:P_log_derivative} gives,
for \(0<t<m+\frac12\),
\begin{equation*}
t\frac{j_m'(t)}{j_m(t)}=t\frac{J_{m+\frac12}'(t)}{J_{m+\frac12}(t)}
-\frac12\ge\frac{\sqrt{(2m+2)^2-4t^2}-2}{2}>0.
\end{equation*}
Thus \(j_m\) is positive and strictly increasing on \((0,m+\frac12)\).  The
spherical Bessel equation also gives
\begin{equation*}
\frac{\mathrm d}{\mathrm dt}\bigl(j_m(t)+t\,j_m'(t)\bigr)
=2j_m'(t)+t\,j_m''(t)=\left(\frac{m(m+1)}{t}-t\right)j_m(t).
\end{equation*}
Choose $m_\tau^{(0)}:=\max\left\{1,\,\left\lceil\frac{1}{2\sqrt{1-b_\tau^2}}\right\rceil
\right\}$. If \(m\ge m_\tau^{(0)}\), then
\begin{equation*}
b_\tau\left(m+\frac12\right)<\sqrt{m(m+1)}.
\end{equation*}
Consequently, \(j_m(t)+t\,j_m'(t)\) is positive and strictly increasing on
\((0,qb_\tau)\), uniformly for \(q<m+\frac12\). Therefore
\begin{align}
R_\tau^{(j)}
&\le \frac{\tau j_m^2(q\tau)}{(b_\tau-a_\tau)j_m^2(qa_\tau)},
\label{eq:TM_ratio_j} \\
R_\tau^{(d)}
&\le \frac{\tau\bigl(j_m(q\tau)+q\tau\,j_m'(q\tau)\bigr)^2}
{(b_\tau-a_\tau)\bigl(j_m(qa_\tau)+qa_\tau\,j_m'(qa_\tau)\bigr)^2}.
\label{eq:TM_ratio_D}
\end{align}
	
Let $P_{m+\frac12}(t):=tJ_{m+\frac12}^2(t)$. The uniform estimate \eqref{eq:J_mz} gives
\begin{equation}\label{eq:TM_P_ratio}
\frac{P_{m+\frac12}(q\tau)}{P_{m+\frac12}(qa_\tau)}
\leq \left(\frac{\tau}{a_\tau}\right)^{(2m+1)c_\tau}.
\end{equation}
Since $j_m(t)=\sqrt{\pi/(2t)}J_{m+\frac12}(t)$, it follows that
\begin{equation}\label{eq:TM_j_ratio}
\frac{j_m^2(q\tau)}{j_m^2(qa_\tau)}
=\left(\frac{a_\tau}{\tau}\right)^2
\frac{P_{m+\frac12}(q\tau)}{P_{m+\frac12}(qa_\tau)}
\leq \left(\frac{a_\tau}{\tau}\right)^2
\left(\frac{\tau}{a_\tau}\right)^{(2m+1)c_\tau}.
\end{equation}
	
Define $H_m(t) := \frac12 + t\,\frac{J_{m+\frac12}'(t)}{J_{m+\frac12}(t)}$. Then
\begin{equation*}
j_m(t) + t\,j_m'(t) = j_m(t)\,H_m(t).
\end{equation*}
The recurrence relation
\begin{equation*}
t\,J_{m+\frac12}'(t)=\left(m+\frac12\right)J_{m+\frac12}(t)-t\,J_{m+\frac32}(t)
\end{equation*}
and the above positivity imply $H_m(t) < m+1$.
Krasikov's lower bound gives
\begin{equation*}
H_m(qa_\tau) \ge \frac12\sqrt{(2m+2)^2 - 4q^2 a_\tau^2} 
\ge \left(m+\frac12\right)c_\tau .
\end{equation*}
Since \(m\ge1\), we obtain
\begin{equation*}
0 < \frac{H_m(q\tau)}{H_m(qa_\tau)}
\le \frac{m+1}{\left(m+\frac12\right)c_\tau}
\le \frac{2}{c_\tau}.
\end{equation*}
Together with \eqref{eq:TM_j_ratio}, this proves
\begin{equation}\label{eq:TM_D_ratio}
	\frac{\bigl(j_m(q\tau)+q\tau\,j_m'(q\tau)\bigr)^2}
	{\bigl(j_m(qa_\tau)+qa_\tau\,j_m'(qa_\tau)\bigr)^2}
	\le
	\frac{4}{c_\tau^2}
	\left(\frac{a_\tau}{\tau}\right)^2
	\left(\frac{\tau}{a_\tau}\right)^{(2m+1)c_\tau}.
\end{equation}
	
Combining \eqref{eq:TM_ratio_j}--\eqref{eq:TM_D_ratio}, we find a constant
\(C_\tau > 0\), depending only on \(\tau\), such that
\begin{equation}\label{eq:TM_uniform_localization_bound}
	1 - \mathcal L_\delta^2(\mathbf E_1)
	\le
	C_\tau
	\left(\frac{\tau}{a_\tau}\right)^{(2m+1)c_\tau}
	\le
	C_\tau \, f \tau^{m+1},
\end{equation}
where \(f\) is defined in \eqref{eq:def_f}. The last inequality follows from
\(2m+1 \ge \frac{3}{2}(m+1)\) and \(0 < \tau/a_\tau < 1\).

Let \(d_\varepsilon = 2\varepsilon - \varepsilon^2 > 0\). Since \(0 < f(\tau) < 1\), the integer
\begin{equation}\label{eq:C_lowerbound_TM}
M_{\mathrm{loc}}^{\mathrm{TM}}(\varepsilon,\delta)
:=
\min\left\{
M \in \mathbb N :
M \ge m_\tau^{(0)}
\ \text{and}\
C_\tau f(\tau)^{m+1} < d_\varepsilon
\ \text{for every } m \ge M
\right\}
\end{equation}
is finite and depends only on \(\varepsilon\) and \(\delta\).

For \(m \ge M_{\mathrm{loc}}^{\mathrm{TM}}(\varepsilon,\delta)\), \eqref{eq:TM_uniform_localization_bound} gives
\begin{equation}\label{eq:TM_localization_step}
	1 - \mathcal L_\delta^2(\mathbf E_1) < 2\varepsilon - \varepsilon^2,
\end{equation}
which implies \(\mathcal L_\delta(\mathbf E_1) > 1 - \varepsilon\).
\end{proof}

We now use Lemma~\ref{lem:key_TM} to count a localized TM subfamily.  Let
\begin{equation*}
m_0 := M_{\mathrm{loc}}^{\mathrm{TM}}(\varepsilon,\delta), \quad
T_m := \min\left\{R,\frac{m+\frac12}{\xi}\right\}, \quad
M_1 := \left\lfloor \xi R - \frac12 \right\rfloor, \quad
M_2 := \left\lfloor R - \frac12 \right\rfloor.
\end{equation*}
For \(m_0 \le m \le M_2\), every TM eigenvalue in \((0,T_m)\) satisfies the hypothesis of Lemma~\ref{lem:key_TM}.  Its entire angular eigenspace, of dimension \(2m+1\), is therefore counted by \(N_{\mathrm{loc}}^{\mathrm{TM}}(R;\varepsilon,\delta)\).  Hence, for all sufficiently large \(R\),
\begin{equation}\label{eq:TM_localized_subfamily_bound}
	N_{\mathrm{loc}}^{\mathrm{TM}}(R;\varepsilon,\delta)
	\ge \sum_{m=m_0}^{M_2} N_m^{\mathrm{TM}}(T_m).
\end{equation}
Since \(\xi T_m \le m+\frac12\) and the first positive zero of \(J_{m+\frac12}\) is larger than \(m+\frac12\), one has \(B_{m+\frac12}(\xi T_m)=0\).  The interval-counting formula \eqref{eq:N_m_TM_decomposition_corrected} consequently gives
\begin{equation*}
N_m^{\mathrm{TM}}(T_m)
= (2m+1)\left[B_{m+\frac12}(T_m) + e_m^{\mathrm{TM}}(T_m)\right], \quad |e_m^{\mathrm{TM}}(T_m)| \le 1.
\end{equation*}
The weighted endpoint errors sum to $\mathcal O(R^2)$.  Moreover, $T_m=(m+\frac12)/\xi$ for $m\leq M_1$, while $T_m=R$ for $m\geq M_1+1$.  Thus \eqref{eq:TM_localized_subfamily_bound} implies
\begin{align*}
N_{\mathrm{loc}}^{\mathrm{TM}}(R;\varepsilon,\delta)
\geq{}&\sum_{m=m_0}^{M_1}(2m+1)
B_{m+\frac12}\!\left(\frac{m+\frac12}{\xi}\right) +\sum_{m=\max\{m_0,M_1+1\}}^{M_2}(2m+1)
B_{m+\frac12}(R)-CR^2.
\end{align*}
Applying the uniform zero-counting estimates \eqref{eq:NB1} and \eqref{eq:NB2}, we find that the total weighted zero-counting remainder is \(\mathcal O(R^2\log R)\).  Consequently,
\begin{align*}
N_{\mathrm{loc}}^{\mathrm{TM}}(R;\varepsilon,\delta)
\geq{}&\sum_{m=0}^{M_1}\frac{2(m+\frac12)^2}{\pi}
\left(\frac{\sqrt{1-\xi^2}}{\xi}-\arccos\xi\right)
+\sum_{m=M_1+1}^{M_2}\frac{2R^2}{\pi}
\Phi\!\left(\frac{m+\frac12}{R}\right)
-CR^2\log R,
\end{align*}
where $\Phi(x)=x\sqrt{1-x^2}-x^2\arccos x$.  These are exactly the two principal sums already evaluated in the TE calculation above.  Substituting those estimates and simplifying their cubic terms yields
\begin{equation*}
N^{\mathrm{TM}}_{\mathrm{loc}}(R;\varepsilon,\delta)
\geq \frac{2R^3}{9\pi}(1-\xi^2)^{3/2}-CR^2\log R.
\end{equation*}
This proves the TM lower bound in Theorem~\ref{thm:lower_bdd}.


\section{Proof of Theorem \ref{thm:upper_bdd}} \label{sec:upper_bound}

To establish the upper bounds, we estimate the complement of the localized family. The crucial point is that whenever both components of a transmission eigenpair have a fixed positive $L^2$-energy in the interior ball $B_{1-\delta}$, neither component can be surface-localized near $\partial\Omega$.

Set $b=1-\delta$. Since $1-\mathcal L_\delta^2(\mathbf E)$ represents the fraction of the $L^2$-energy of $\mathbf E$ located in the interior ball $B_b$, our task reduces to estimating the relevant interior energy ratios. For a TE or TM eigenpair of angular order $m$, define
\begin{equation*}
\mu=m+\frac12,\quad x_1=\frac{k\xi}{\mu},\quad x_2=\frac{k}{\mu}.
\end{equation*}
The radial arguments of the two components are $\mu x_1 r$ and $\mu x_2 r$, respectively. Consequently, their complementary energy ratios are given by $Q_\mu^\sigma(x_1;b)$ and $Q_\mu^\sigma(x_2;b)$, where $\sigma\in\{\mathrm{TE},\mathrm{TM}\}$ and, for $x>1/b$, we set
\begin{equation*}
Q_\mu^\sigma(x;b) := \frac{\int_0^b \mathcal I_\sigma(r;\mu x)\,\mathrm d r}{\int_0^1 \mathcal I_\sigma(r;\mu x)\,\mathrm d r},
\end{equation*}
with the integrands
\begin{equation*}
\mathcal I_{\mathrm{TE}}(r;\mu x) := r^2 j_m^2(\mu x r),
\end{equation*}
and
\begin{equation*}
\mathcal I_{\mathrm{TM}}(r;\mu x) :=
m^2(m+1)^2 j_m^2(\mu x r)
+ m(m+1)\bigl(j_m(\mu x r)+\mu x r\, j_m'(\mu x r)\bigr)^2.
\end{equation*}

The next lemma shows that, uniformly away from the turning point, both
quotients are governed by the same explicit profile.
\begin{lemma}\label{lem:uniform_interior_energy}
For any \(0<b<1\) and \(x_0>1/b\), we have
\begin{equation*}
\sup_{x\ge x_0}\Bigl(\bigl|Q_\mu^{\mathrm{TE}}(x;b)-\Gamma_b(x)\bigr|
+\bigl|Q_\mu^{\mathrm{TM}}(x;b)-\Gamma_b(x)\bigr|\Bigr)\to 0,
\quad \mbox{as~} \mu\to\infty,
\end{equation*}
where $\Gamma_b(x):=\sqrt{\frac{b^2x^2-1}{x^2-1}}$.
\end{lemma}

\begin{proof}
We begin by rewriting the energy integrals as endpoint quantities. This allows us to avoid integrating a pointwise oscillatory expansion through the turning region. Set
\begin{equation*}
I_\mu(z):=\int_0^z t\, J_\mu(t)^2\,\mathrm d t .
\end{equation*}
By the Lommel identity together with the standard Bessel recurrences, we have
\begin{equation*}
I_\mu(z)=\frac12\left[(z^2-\mu^2)J_\mu(z)^2+z^2 J_\mu'(z)^2\right].
\end{equation*}
Since \(j_m(t)=\sqrt{\pi/(2t)}\,J_\mu(t)\), the change of variables \(t=\mu x r\) yields
\begin{equation*}
Q_\mu^{\mathrm{TE}}(x;b)=\frac{I_\mu(\mu x b)}{I_\mu(\mu x)}.
\end{equation*}
	
For the TM quotient, the same change of variables gives
\begin{equation*}
Q_\mu^{\mathrm{TM}}(x;b)=\frac{K_\mu(\mu x b)}{K_\mu(\mu x)},
\end{equation*}
where
\begin{equation*}
K_\mu(z):=\int_0^z\left[\left(\mu^2-\frac14\right)j_m(t)^2
+\bigl((t j_m(t))'\bigr)^2\right]\,\mathrm d t .
\end{equation*}
Using the Bessel equation together with the identity
\begin{equation*}
(t j_m(t))'=\sqrt{\frac{\pi}{2t}}\left(\frac12 J_\mu(t)+t J_\mu'(t)\right),
\end{equation*}
one obtains
\begin{equation*}
K_\mu(z)=\frac{\pi}{2}\left[z J_\mu(z) J_\mu'(z)+I_\mu(z)+\frac12 J_\mu(z)^2\right].
\end{equation*}
	
For each fixed \(y_0>1\), the classical large-order oscillatory expansions of \(J_\mu\) and \(J_\mu'\) are uniform on \([y_0,\infty)\); see \cite{OL74}. Inserting these expansions into the endpoint identities above yields
\begin{equation*}
I_\mu(\mu y)=\frac{\mu}{\pi}\sqrt{y^2-1}\left(1+\mathcal O_{y_0}(\mu^{-1})\right),
\end{equation*}
and
\begin{equation*}
K_\mu(\mu y)=\frac{\mu}{2}\sqrt{y^2-1}\left(1+\mathcal O_{y_0}(\mu^{-1})\right),
\end{equation*}
uniformly for \(y\ge y_0\).
	
Now choose \(y_0=b x_0>1\). For every \(x\ge x_0\), both \(b x\) and \(x\) belong to \([y_0,\infty)\). Consequently,
\begin{equation*}
Q_\mu^{\mathrm{TE}}(x;b)=\sqrt{\frac{b^2x^2-1}{x^2-1}}+o(1),\quad
Q_\mu^{\mathrm{TM}}(x;b)=\sqrt{\frac{b^2x^2-1}{x^2-1}}+o(1),
\end{equation*}
uniformly for \(x\ge x_0\). This completes the proof.
\end{proof}
	
We now turn this approximation into a spectral cutoff. In particular, whenever \(k\xi/\mu\) is larger than a suitable threshold, the explicit profile \(\Gamma_b\) implies that both interior energy ratios admit a positive lower bound.

\begin{lemma}\label{lem_inside_estimate}
Fix $0<\xi<1$, $0<\delta<1$, and $0<\varepsilon_0<1-\delta$. Choose $\eta\in(\varepsilon_0,1-\delta)$ and set
\begin{equation}\label{eq:def_a}
	a_{\eta,\delta}:=\sqrt{\frac{(1-\delta)^2-\eta^2}{1-\eta^2}}.
\end{equation}
Let $k$ be a TE or TM transmission eigenvalue of angular order $m$, and let $(\mathbf E_1,\mathbf E_2)$ be a corresponding eigenpair. If
\begin{equation}\label{eq:assum_1}
	k>\frac{\mu}{\xi a_{\eta,\delta}}
	=\frac{\mu}{\xi}
	\sqrt{\frac{1-\eta^2}{(1-\delta)^2-\eta^2}},
\end{equation}
then, for all sufficiently large $m$ (uniformly in the polarization), we have
\begin{equation}\label{eq:inside_estimate_1}
	1-\mathcal L_\delta^2(\mathbf E_1)>\varepsilon_0,	\quad
	1-\mathcal L_\delta^2(\mathbf E_2)>\varepsilon_0.
\end{equation}
The required lower bound on $m$ depends only on $\varepsilon_0,\eta,\delta$, and is independent of $k$ and $\xi$.
\end{lemma}

\begin{proof}
Let \(b=1-\delta\) and \(a=a_{\eta,\delta}\). A direct substitution into the definition of \(\Gamma_b\) gives \(\Gamma_b(1/a)=\eta\). Moreover,
\begin{equation*}
\frac{\mathrm d}{\mathrm dx}\Gamma_b(x)^2
=\frac{2x(1-b^2)}{(x^2-1)^2}>0,
\quad x>1/b,
\end{equation*}
so \(\Gamma_b\) is strictly increasing on its domain. Under the hypothesis \eqref{eq:assum_1},
\begin{equation*}
x_1:=\frac{k\xi}{\mu}>\frac1a>\frac1b,\quad
x_2:=\frac{k}{\mu}=\frac{x_1}{\xi}>x_1.
\end{equation*}
For each polarization \(\sigma\in\{\mathrm{TE},\mathrm{TM}\}\), the complementary interior energy ratios of \(\mathbf E_1\) and \(\mathbf E_2\) are respectively \(Q_\mu^\sigma(x_1;b)\) and \(Q_\mu^\sigma(x_2;b)\). By Lemma~\ref{lem:uniform_interior_energy}, for all sufficiently large \(m\), both ratios differ from their corresponding \(\Gamma_b\)-values by less than \(\eta-\varepsilon_0\). Since \(x_i>1/a\) and \(\Gamma_b\) is strictly increasing, we have \(\Gamma_b(x_i)>\Gamma_b(1/a)=\eta\). Consequently,
\begin{equation*}
Q_\mu^\sigma(x_i;b)>\Gamma_b(x_i)-(\eta-\varepsilon_0)>\eta-(\eta-\varepsilon_0)
=\varepsilon_0,\quad i=1,2.
\end{equation*}
Thus both inequalities in \eqref{eq:inside_estimate_1} hold, completing the proof.
\end{proof}

We now apply Lemma~\ref{lem_inside_estimate} with the localization tolerance $\varepsilon$. Define $d_\varepsilon:=1-(1-\varepsilon)^2=2\varepsilon-\varepsilon^2$.
Fix $\eta\in(d_\varepsilon,1-\delta)$, and let $a=a_{\eta,\delta}$ be given by \eqref{eq:def_a}. Lemma~\ref{lem_inside_estimate} then provides an integer $m_*$, independent of $k$, $R$, and $\xi$, such that for both polarizations $\sigma\in\{\mathrm{TE},\mathrm{TM}\}$, whenever $m\ge m_*$ and $k>\mu/(\xi a)$, we have
\begin{equation}\label{eq:upper_nonloc_condition}
	1-\mathcal L_\delta^2(\mathbf E_i)>d_\varepsilon,\quad i=1,2,
\end{equation}
which precisely means that neither component is $(\varepsilon,\delta)$-localized.

For a fixed polarization \(\sigma\in\{\mathrm{TE},\mathrm{TM}\}\), let \(F_m^\sigma\) denote the corresponding characteristic function and \(N_m^\sigma\) the modal counting function. For \(m\ge m_*\), we define the non-localized modal count of order \(m\) by
\begin{equation*}
\widetilde S_m^\sigma(R;\varepsilon,\delta)
:=(2m+1)\#\left\{k:F_m^\sigma(k)=0,\ \frac{\mu}{\xi a}<k<R\right\},
\end{equation*}
and the total non-localized count by
\begin{equation*}
\widetilde N^\sigma(R;\varepsilon,\delta)
:=
\sum_{m=m_*}^{\infty}
\widetilde S_m^\sigma(R;\varepsilon,\delta).
\end{equation*}
In view of \eqref{eq:upper_nonloc_condition}, these eigenvalues are not included in the localized count \(N_{\mathrm{loc}}^\sigma\). Consequently,
\begin{equation}\label{eq:upper_complement_ineq}
	N_{\mathrm{loc}}^\sigma(R;\varepsilon,\delta)
	\leq
	N^\sigma(R)-\widetilde N^\sigma(R;\varepsilon,\delta).
\end{equation}

It remains to determine the maximal angular order for which the interval \((\mu/(\xi a), R)\) is nonempty. Define
\begin{equation*}
M_R:=\max\left\{m\in\mathbb N:m\geq m_*,~\frac{\mu}{\xi a}<R \right\}.
\end{equation*}
Since \(\mu=m+\tfrac12\), the condition \(\mu/(\xi a)<R\) is equivalent to \(m<\xi aR-\tfrac12\). Hence, ignoring the fixed lower cutoff \(m_*\), we have the asymptotic formula
\begin{equation*}
M_R=\left\lfloor \xi aR-\frac12\right\rfloor+\mathcal O(1).
\end{equation*}

For \(m>M_R\), the set in the definition of \(\widetilde S_m^\sigma\) is empty, so the sum defining \(\widetilde N^\sigma\) effectively terminates at \(M_R\).
For $m\leq M_R$, we have
\begin{equation*}
\widetilde S_m^\sigma(R;\varepsilon,\delta)
=N_m^\sigma(R)-N_m^\sigma\!\left(\frac{\mu}{\xi a}\right)+\mathcal O(2m+1),
\end{equation*}
where the error only accounts for the possible endpoint
$k=\mu/(\xi a)$.  Summing the endpoint errors and applying the bound
\begin{equation*}
\sum_{m\leq M_R}(2m+1)=\mathcal O(R^2),
\end{equation*}
we obtain, uniformly in both polarizations,
\begin{eqnarray}  \label{eq:NTEuc}
\widetilde N^\sigma(R;\varepsilon,\delta)
&=&\sum_{m=1}^{\lfloor \xi aR-\frac12\rfloor}
(2m+1)\left[B_\mu(R)-B_\mu(\xi R)\right]-\sum_{m=1}^{\lfloor \xi aR-\frac12\rfloor}
(2m+1)\left[B_\mu\!\left(\frac{\mu}{\xi a}\right)-B_\mu\!\left(\frac{\mu}{a}\right)
\right]+\mathcal O(R^2) \nonumber\\
&=& I_1-I_2+\mathcal O(R^2),
\end{eqnarray}
where \(I_1\) accounts for the contribution from the full cutoff \(R\), whereas \(I_2\) removes the modes lying below the non-localization threshold \(\mu/(\xi a)\).

For the term \(I_1\), define
\begin{equation*}
F_\xi(x):=x\sqrt{1-x^2}-x^2\arccos x-x\sqrt{\xi^2-x^2}+x^2\arccos(x/\xi).
\end{equation*}
Then, applying the Euler–Maclaurin summation formula to the resulting sum, we obtain
\begin{equation*}
I_1=\frac{2R^2}{\pi}\sum_{m=1}^{\lfloor \xi aR-\frac12\rfloor}
F_\xi\!\left(\frac{m+\frac12}{R}\right)+\mathcal O(R^2\log R).
\end{equation*}
Since \(F_\xi\in C^1([0,\xi a])\), the above sum is a Riemann sum for the integral of \(F_\xi\), and therefore
\begin{equation*}
I_1=\frac{2R^3}{\pi}\int_0^{\xi a}F_\xi(x)\,\mathrm dx+\mathcal O(R^2\log R).
\end{equation*}
Now define
\begin{equation*}
P_1(x):=\int_0^x\left(t\sqrt{1-t^2}-t^2\arccos t\right)\,\mathrm dt
=\frac19+\frac{4x^2-1}{9}\sqrt{1-x^2}-\frac13x^3\arccos x.
\end{equation*}
Making the substitution \(x=\xi y\) in the integral yields
\begin{equation}\label{eq:NTEucI1}
I_1=\frac{2}{\pi}\left[	P_1(\xi a)-\xi^3P_1(a)\right]R^3+\mathcal O(R^2\log R).
\end{equation}

For the term \(I_2\), define
\begin{equation*}
P_2(x):=\frac{\sqrt{1-x^2}}{x}-\arccos x,\quad 0<x<1.
\end{equation*}
We invoke the uniform Bessel zero-counting estimate from \cite{JLZZ25}: for every fixed \(c\in(0,1)\) and all \(0\le \mu\le cT\),
\begin{equation*}
B_\mu(T)=\frac1\pi\left[\sqrt{T^2-\mu^2}-\mu\arccos\!\left(\frac{\mu}{T}\right)\right]
+\mathcal O(\log T),
\end{equation*}
with the implied constant depending only on \(c\). In the present context, we have
\begin{equation*}
\frac{\mu}{R}\le \xi a,\quad\frac{\mu}{\xi R}\le a,\quad 0<\xi a<a<1,
\end{equation*}
so the estimate is applicable uniformly in the summation range. Consequently,
\begin{equation*}
B_\mu(\mu/x)=\frac{\mu}{\pi}P_2(x)+\mathcal O(\log R),
\end{equation*}
uniformly for \(x=\xi a\) and \(x=a\). Since
\begin{equation*}
\sum_{m=1}^{\lfloor \xi aR-\frac12\rfloor}2\mu^2
=\frac{2}{3}(\xi aR)^3+\mathcal O(R^2),
\end{equation*}
we obtain
\begin{equation}\label{eq:NTEucI2}
I_2=\frac{2}{3\pi}(\xi a)^3\left[P_2(\xi a)-P_2(a)\right]R^3+\mathcal O(R^2\log R).
\end{equation}

Combining \eqref{eq:NTEuc}–\eqref{eq:NTEucI2} with the Weyl laws \eqref{N^TE(R)} and \eqref{eq:N^TM(R)}, and then applying the complement inequality \eqref{eq:upper_complement_ineq}, yields
\begin{align*}
N_{\mathrm{loc}}^\sigma(R;\varepsilon,\delta)
&\leq\frac{2R^3}{9\pi}(1-\xi^3)-\frac{2R^3}{\pi}\left[P_1(\xi a)-\xi^3P_1(a)\right]\\
&+\frac{2R^3}{3\pi}(\xi a)^3\left[P_2(\xi a)-P_2(a)\right]+\mathcal O(R^2\log R).
\end{align*}
Using the elementary identity
\begin{equation*}
\frac{2}{9\pi}-\frac{2}{\pi}P_1(x)+\frac{2}{3\pi}x^3P_2(x)=
\frac{2}{9\pi}(1-x^2)^{3/2},
\end{equation*}
we can simplify the leading-order coefficient to obtain
\begin{equation*}
N_{\mathrm{loc}}^\sigma(R;\varepsilon,\delta)
\le \frac{2}{9\pi}\left[(1-\xi^2a^2)^{3/2}-\xi^3(1-a^2)^{3/2}\right]R^3+C R^2\log R,
\quad \sigma\in\{\mathrm{TE},\mathrm{TM}\},
\end{equation*}
where the constant \(C>0\) is independent of \(R\), \(\xi\), and the polarization. This proves Theorem~\ref{thm:upper_bdd}.


\section{Numerical examples}\label{sec:numerics}
In this section, we present several representative numerical examples to illustrate surface localization and to explore whether the ball-based asymptotic picture persists numerically for other geometries. All simulations were performed on an i7-13700 processor with 32 GB of RAM.

The tested domains include three three-dimensional geometries:

\begin{enumerate}
\item[(a)] The unit ball in $\mathbb{R}^3$.
\item[(b)] A cube with side length $a_{\mathrm{cube}} = 1.612$.
\item[(c)] A kite-shaped domain parameterized by
\begin{eqnarray}\label{eq:kite}
x(t,\theta) &=& \cos t + 0.65\cos(2t) - 0.65, \nonumber\\
y(t,\theta) &=& \sin t \cos \theta, \\
z(t,\theta) &=& \sin t \sin \theta, \nonumber
\end{eqnarray}
where $t \in [0, \pi]$ and $\theta \in [0, 2\pi]$.
\end{enumerate}
The volumes of these geometries are each approximately $4.1888$; this common-volume choice is used only to normalize the numerical comparison.  The ball Weyl law proved above is not invoked here as a theorem for the nonspherical domains. Following \cite{DLWW22,MS12}, we use edge elements to approximate \eqref{eq:trans2}, and the eigenvalues and eigenvectors for the corresponding nonsymmetric eigenvalue problem are computed with MATLAB sparse eigenvalue routines in the PDE implementation. These computations are intended as qualitative numerical evidence rather than as a mesh-convergence study.
We set $\varepsilon = 0.1$ and $\delta = 0.1$ in Definition~\ref{def:localization}.

First, in Figure~\ref{fig:total}, we present the $\mathbf{E}_1$-component of the eigenfunctions for three geometries. These examples correspond to the largest surface-localizing eigenvalue with $k^2<16$ in the tested range. The surface-localizing property holds robustly across these geometric configurations. The eigenfunctions exhibit strong concentration near the boundaries and negligible interior activity, thereby confirming the surface-localization phenomenon.

\begin{figure}[htbp]
	\centering
	\subfigure[]{
		\includegraphics[width=3.5cm,height=3cm]{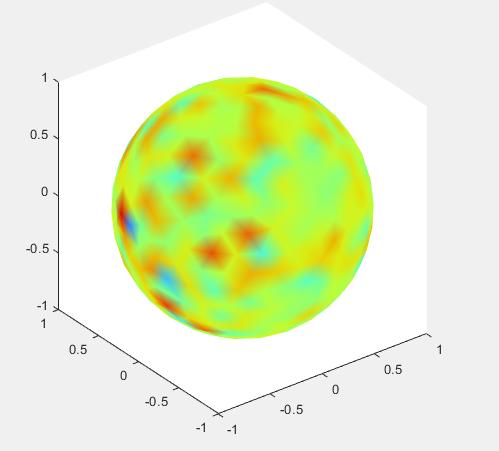}
	}\quad
	\subfigure[]{
		\includegraphics[width=3.5cm,height=3cm]{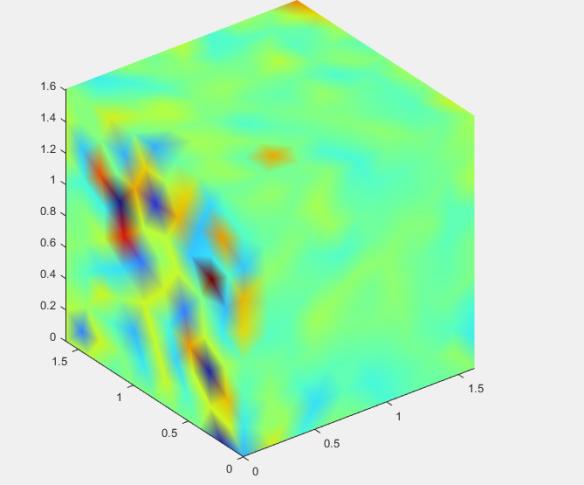}
	}\quad
	\subfigure[]{
		\includegraphics[width=3.5cm,height=3cm]{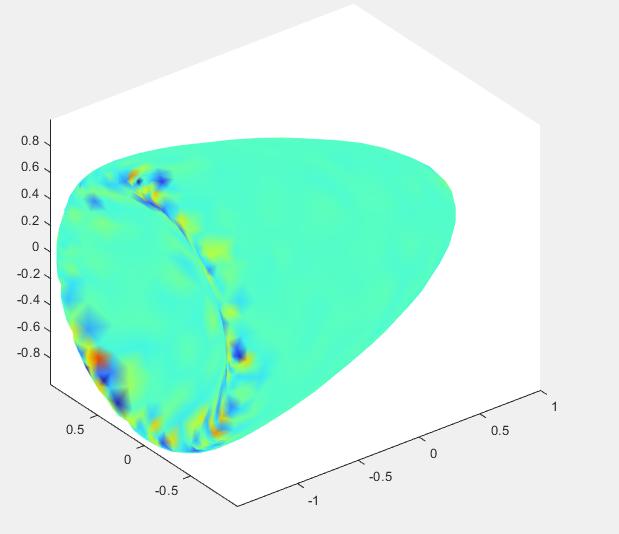}
	}

	\subfigure[]{
		\includegraphics[width=3.5cm,height=3cm]{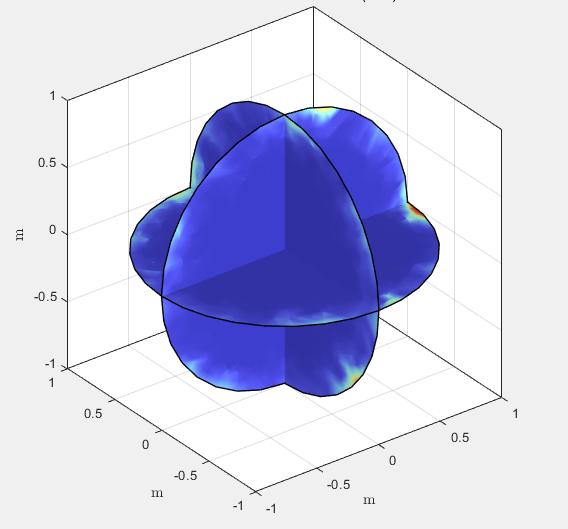}
	}
	\quad
	\subfigure[]{
		\includegraphics[width=3.5cm,height=3cm]{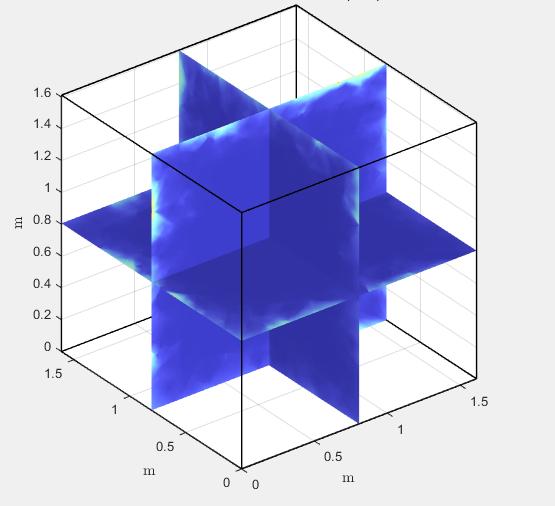}
	}
	\quad
	\subfigure[]{
		\includegraphics[width=3.5cm,height=3cm]{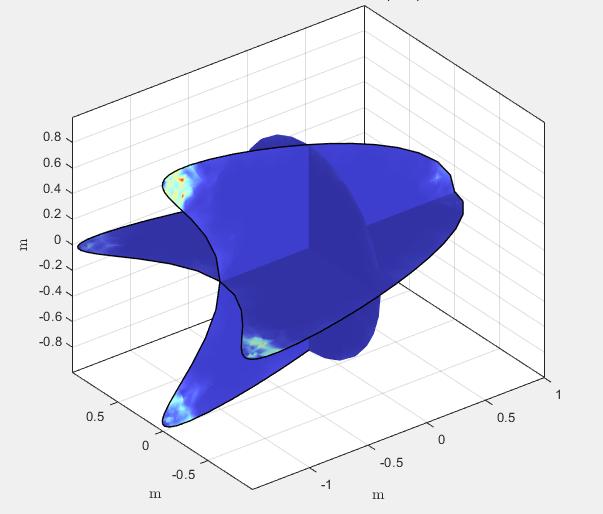}
	}
	\caption{Surface-localized eigenfunctions for three geometries: (a) the unit ball; (b) a cube with $a_{\mathrm{cube}}=1.612$; (c) the kite-shaped domain parameterized by \eqref{eq:kite}; (d) a slice of the unit ball; (e) a slice of the cube; (f) a slice of the kite-shaped domain.}
	\label{fig:total}
\end{figure}

For the unit ball, the localization ratios are evaluated from
\eqref{energy_E1} for TE modes and \eqref{eq:TM W_E_1} for TM modes.
For $\sigma\in\{\mathrm{TE},\mathrm{TM}\}$ we define the finite-cutoff
localized modal proportion by
\begin{equation*}
\rho_{\varepsilon,\delta}^{\sigma}(R)
:=\frac{N_{\mathrm{loc}}^{\sigma}(R;\varepsilon,\delta)}
{N^{\sigma}(R)}.
\end{equation*}
Figure~\ref{fig:tefm} displays these TE and TM modal proportions for
the unit ball.  The theoretical curves are the ratios of the leading
cubic coefficients in Theorems~\ref{thm:lower_bdd} and
\ref{thm:upper_bdd} to the full Weyl coefficient in
Theorem~\ref{thm:weyl}.  Because the remainder constants are not
specified, the finite-cutoff data are presented as being consistent
with the asymptotic bounds, rather than as a strict finite-$R$
verification.
 \begin{figure}[htbp]
 	\centering
 	\includegraphics[width=10cm,height=8cm]{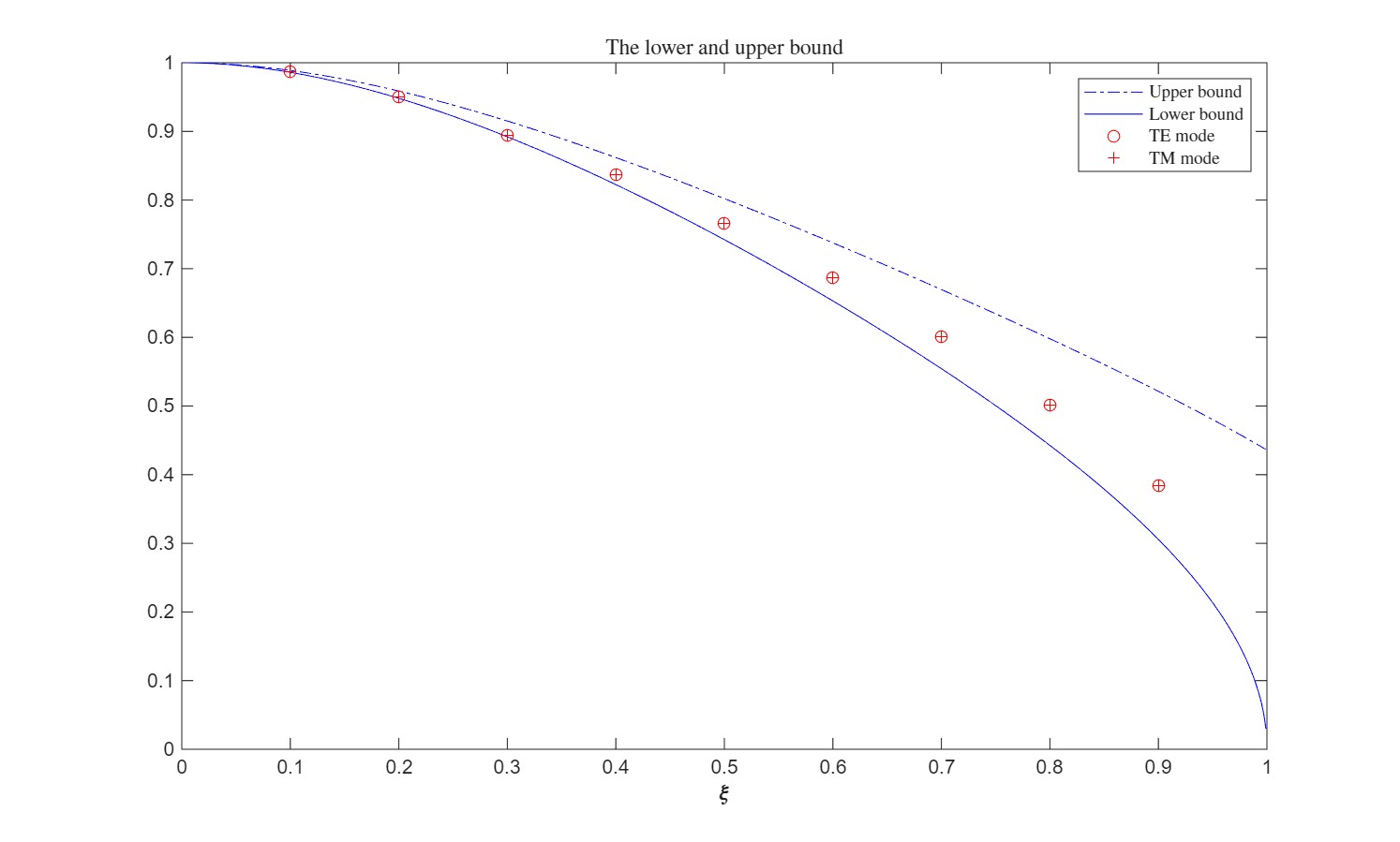}
 	\caption{Spectral distribution of surface-localized eigenmodes for the unit ball.}
 	\label{fig:tefm}
 \end{figure}

For the other shapes, including the cube and kite-shaped domain, we have no explicit formula for $\mathcal{L}_{\delta}$ as we do for the unit ball; therefore, we provide numerical evidence for the behavior predicted by Theorems~\ref{thm:lower_bdd} and~\ref{thm:upper_bdd}. To this end, we ensure that the volumes of all test domains are comparable. For $\xi>1$, the comparison with the theoretical bounds is understood through the reduction $k^*=\xi k$ and $\xi^*=1/\xi$ described above; the tables below report the original physical refractive index $\xi$. When the refractive index is set to $\xi=6$, Table~\ref{tab:1} reports the decimal localized-to-total modal proportion, rounded to four decimal places. Each discrete eigenvector is counted once; repeated eigenvalues and numerically split eigenvalue clusters are therefore counted with their computed multiplicity, in agreement with the modal counting convention used in the analysis. Thus $0.0000$ means that no localized mode was found in the corresponding computed sample. The proportions are comparable and increase as the spectral range of $k^2$ grows for the unit ball, the cube, and the kite-shaped domain, making the finite-cutoff trend directly visible. Notably, the kite geometry exhibits the highest ratio among the three domains.

 \begin{table}[htbp]
 	\centering
 	\begin{tabular}{|c|c|c|c|}
 		\hline
 		$k^2$  & Ball     & Cube    & Kite        \\
 		\hline \hline
 		$[0,50]$   & 0.0000   & 0.0000   & 0.0000    \\
 		\hline
 		$[0,100]$  & 0.0000   & 0.0000   & 0.0119    \\
 		\hline
 		$[0,150]$  & 0.0165   & 0.0181   & 0.0470  \\
 		\hline
 		$[0,200]$  & 0.0366   & 0.0403   & 0.0618   \\
 		\hline
 		$[0,256]$  & 0.0647   & 0.0691   & 0.0742   \\
 		\hline
 	\end{tabular}~\\~
	\caption{The decimal localized-to-total modal proportion for different spectral ranges of $k^2$ with $\xi=6$.}
 	\label{tab:1}
 \end{table}

Furthermore, for a fixed spectral range $k^2\in[0,256]$, Table~\ref{tab:2} gives the same four-decimal modal proportion as a function of the refractive index $\xi$. The proportion remains zero for small $\xi$ ($\xi=2,3$), then becomes positive and increases with $\xi$, displaying the transition toward stronger surface localization for the unit ball, the cube, and the kite-shaped domain. Among the three geometries, the kite again shows the highest sensitivity.
 \begin{table}[htbp]
 	\centering
 	\begin{tabular}{|c|c|c|c|}
 		\hline
 		$\xi$ & Ball     & Cube     & Kite         \\
 		\hline \hline
 		\  2    & 0.0000   & 0.0000    & 0.0000       \\
 		\hline
 		\  3    & 0.0000   & 0.0000    & 0.0000       \\
 		\hline
 		\  4   & 0.0104   & 0.0174    & 0.0086      \\
 		\hline
 		\  5   & 0.0508   & 0.0552    & 0.0405    \\
 		\hline
 		\  6   & 0.0647   & 0.0691    & 0.0742    \\
 		\hline
 	\end{tabular}~\\~
	\caption{The decimal localized-to-total modal proportion for different refractive indices $\xi$ with spectral range $k^2\in[0,256]$.}
 	\label{tab:2}
 \end{table}
 
 
\section*{Acknowledgment}
Y. Jiang was supported in part by China's National Key R\&D Program (2024YFA1012302), and the RGC Project JRFS2627-1S06.
The work of H. Liu is supported by the Hong Kong RGC General Research Funds (Projects 11311122, 11300821, and 11304224), the NSF/RGC Joint Research Fund (Project N\_CityU101/21), and the ANR/RGC Joint Research Fund (Project A\_CityU203/19). The work of K. Zhang is supported in part by the National Natural Science Foundation of China (Grant No. 12271207).


\end{document}